\documentclass[review, compress, 3p]{elsarticle}
\usepackage{lineno,hyperref}
\usepackage{amssymb,amsfonts}
\usepackage{amsmath,amsthm}
\usepackage{algorithm,algorithmic}
\usepackage{graphicx,subfigure}
\usepackage{float,rotating}
\usepackage{booktabs}
\usepackage{multirow,multicol,lscape}
\usepackage{epstopdf}
\usepackage{color}
\usepackage{setspace}
\usepackage{natbib}
\usepackage{bm}
\usepackage{url,doi}

\makeatletter
\@addtoreset{equation}{section}
\makeatother

\newtheorem{theorem}{Theorem}[section]

\newtheorem{property}{Property}
\newtheorem{assumption}{Assumption}

\journal{Elsevier}

\begin{document}

\begin{frontmatter}

\title{A low-rank Lie-Trotter splitting approach for nonlinear fractional complex Ginzburg-Landau equations}

\author[address1]{Yong-Liang Zhao}
\ead{ylzhaofde@sina.com}

\author[address2]{Alexander Ostermann}

\ead{alexander.ostermann@uibk.ac.at}

\author[address3]{Xian-Ming Gu\corref{correspondingauthor}}
\cortext[correspondingauthor]{Corresponding author}
\ead{guxianming@live.cn}

\address[address1] {School of Mathematical Sciences, \\
University of Electronic Science and Technology of China, \\
Chengdu, Sichuan 611731, P.R. China}
\address[address2]{Department of Mathematics, University of Innsbruck, Technikerstra{\ss}e 13, Innsbruck 6020, Austria}
\address[address3] {School of Economic Mathematics/Institute of Mathematics, \\
Southwestern University of Finance and Economics, Chengdu, Sichuan 611130, P.R. China}

\begin{abstract}
Fractional Ginzburg-Landau equations as the generalization of the classical one
have been used to describe various physical phenomena.
In this paper, we propose a numerical integration method for solving 
space fractional Ginzburg-Landau equations based on a dynamical low-rank approximation.
We first approximate the space fractional derivatives by using a fractional centered difference method.
Then, the resulting matrix differential equation is split into a stiff linear part and a nonstiff (nonlinear) one.
For solving these two subproblems, a dynamical low-rank approach is used.
The convergence of our method is proved rigorously.
Numerical examples are reported which show that the proposed method is robust and accurate. 
\end{abstract}


\end{frontmatter}


\section{Introduction}
\label{sec1}

In physics, the classical complex Ginzburg-Landau model, originally derived by Newell and Whitehead \cite{newell1969finite, lange1974stability},
describes the amplitude evolution of waves in dissipative systems of fluid mechanics close to instabilities.
This model has remarkable success in describing numerous phenomena including
B\'{e}nard convection \cite{newell1969finite, segel1969distant}, plane Poiseuille flow \cite{stewartson1971non} 
and superconductivity \cite{du1992analysis, chapman1992macroscopic}.

Tarasov and Zaslavsky \cite{tarasov2005fractional, tarasov2006fractional} first proposed a fractional generalization of the Ginzburg-Landau equation for fractal media.
Later, fractional Ginzburg-Landau equations (FGLEs) have been used to describe various physical phenomena, 
refer e.g., to \cite{milovanov2005fractional, mvogo2016localized, tarasov2006psi}.
Many properties about FGLEs such as well-posedness, long-time dynamics and asymptotic behavior were studied;
see \cite{pu2013well, guo2013well, lu2013asymptotic, millot2015fractional} and references therein.
However, it is usually not possible to obtain analytical solutions of FGLEs, due to the non-local property of fractional derivatives.
Thus, the development of efficient and reliable numerical techniques for solving FGLEs attracts many researchers. 
Wang and Huang \cite{wang2016implicit} proposed an implicit midpoint difference scheme for FGLEs.
Hao and Sun \cite{hao2017linearized} proposed a linearized high-order difference scheme for the one-dimensional complex FGLE.
Li and Huang \cite{li2019efficient} developed an implicit difference scheme for solving the coupled nonlinear FGLEs.
In \cite{zhang2020exponential}, the authors used an exponential Runge-Kutta method to solve the two-dimensional (2D) FGLE.
Other related work can be found in \cite{zhang2019linearized, li2017galerkin, wang2018efficient, he2018unconditionally, zhang2020linearized} and references therein.

To the best of our knowledge, there are few numerical methods for solving FGLEs by using a dynamical low-rank approximation has been considered.
Thus, in this paper, we mainly study a dynamical low-rank approximation \cite{koch2007dynamical} for solving the following 2D complex FGLE:
\begin{equation}
\begin{cases}
\partial_t u - (\nu + \mathrm{\mathbf{i}} \eta) (\partial_x^\alpha + \partial_y^\beta) u 
+ (\kappa + \mathrm{\mathbf{i}} \xi) |u|^2 u - \gamma u = 0, & (x,y,t) \in \Omega \times (0, T],\\
u(x,y,0) = u_0(x,y), & (x,y) \in \bar{\Omega} = \Omega \cup \partial \Omega, \\
u(x,y,t) = 0, & (x,y) \in \partial \Omega,
\end{cases}
\label{eq1.1}
\end{equation}
where $\mathrm{\mathbf{i}} = \sqrt{-1}$, $\nu > 0,~\kappa > 0,~\eta,~\xi$ are real numbers, $1 < \alpha,~\beta <2$,
$\Omega = (x_L, x_R) \times (y_L, y_R) \subset \mathbb{R}^2$, $u_0(x,y)$ is a given complex function,
$\partial \Omega$ is the boundary of $\Omega$, and
$\gamma \in \mathbb{R}$ is the coefficient of the linear evolution term.
$\partial_x^\alpha$ and $\partial_y^\beta$ denote the Riesz fractional derivatives \cite{gorenflo1998random} 
in the $x$ and $y$ directions, respectively. 
More precisely, they are defined as follows:
\begin{equation*}
\begin{split}
& \partial_x^\alpha u(x,y,t) = -\frac{1}{2 \cos(\alpha \pi/2) \Gamma(2 - \alpha)} \frac{\partial^2}{\partial x^2}
\int_{- \infty}^{\infty} \left|x - \zeta \right|^{1 - \alpha} u(\zeta,y,t) d \zeta, \\
& \partial_y^\beta u(x,y,t) = -\frac{1}{2 \cos(\beta \pi/2) \Gamma(2 - \beta)} \frac{\partial^2}{\partial y^2}
\int_{- \infty}^{\infty} \left|y - \zeta \right|^{1 - \beta} u(x,\zeta,t) d \zeta.
\end{split}
\end{equation*}
Moreover, if $\nu = \kappa = \gamma = 0$, Eq.~\eqref{eq1.1} becomes a nonlinear fractional Schr{\"o}dinger equation \cite{wang2015energy}.
Some numerical methods for such equations can be found in \cite{zhao2014fourth, wang2014linearly, zhai2019error, li2020relaxation}.

Dynamical low-rank approximation was first studied by Koch and Lubich in \cite{koch2007dynamical}.
It is a differential equation-based approach and mainly follows the idea of the Dirac-Frenkel variational principle
\cite{dirac1930principles, lubich2008quantum}.
The aim of this approach is to find low-rank approximations to time-dependent large data matrices or to solutions of large matrix differential equations.
This dynamical low-rank approach yields a differential equation for the approximation matrix on the low-rank manifold.
Nonnenmacher and Lubich \cite{nonnenmacher2008dynamical} first used this approach in different numerical applications,
such as the compression of series of images and a blow-up problem of a reaction-diffusion equation.
Lubich and Oseledets \cite{lubich2014projector} proposed and analyzed an inexpensive integrator (called projector-splitting integrator)
to numerically solve the differential equation arising from the dynamical low-rank approximation.
Later, Kieri et al.~\cite{kieri2016discretized} provided a comprehensive error analysis for the integration method.
Note that their error bounds depend on the Lipschitz constant of the right-hand side of the considered matrix differential equation.
Later, Ostermann et al.~\cite{ostermann2019convergence} proposed a integration method that yields low-rank approximations for stiff matrix differential equations.
The key idea of their method is to first split the equation into stiff and nonstiff parts, and then to follow the dynamical low-rank approach.
Their error analysis shows that the approach is independent of the stiffness and robust with respect to possibly small singular values in the approximation matrix.
More dynamical low-rank based discretizations of high-dimensional time-dependent problems can be found in
\cite{koch2010dynamical, lubich2013dynamical, einkemmer2018low, einkemmer2019quasi, einkemmer2019low, grasedyck2013literature}.

The rest of this paper is organized as follows.
In Section \ref{sec2}, we propose our low-rank approximation for solving Eq.~\eqref{eq1.1}.
The error analysis of our method is given in Section \ref{sec3}.
Two numerical examples provided in Section \ref{sec5} strongly support the theoretical results.
Concluding remarks are given in Section \ref{sec6}.
\section{A low-rank approximation of the 2D complex FGLE}
\label{sec2}

In this section, our low-rank approximation for solving Eq.~\eqref{eq1.1} is derived.
Such an approximation is usually based on a matrix differential equation.
Thus, Eq.~\eqref{eq1.1} is first discretized in space by an appropriate difference method and the resulting equations are then reformulated as a matrix differential equation.

\subsection{The matrix differential equation}
\label{sec2.1}

For two given positive integers $N_x$ and $N_y$, let $h_x = \frac{x_R - x_L}{N_x}$ and $h_y = \frac{y_R - y_L}{N_y}$.
Then, the space can be discretized by $\Omega_h = \left\{ (x_i,y_j) \mid 0 \leq i \leq N_x, 0 \leq j \leq N_y \right\}$.
Let $u_{i j}(t)$ be the numerical approximation of $u(x_i, y_j, t)$. 
For approximating $\partial_x^\alpha u$ and $\partial_y^\beta u$ in Eq.~\eqref{eq1.1},
we choose the second-order fractional centered difference method proposed in \cite{ccelik2012crank}.
Then, we have 
\begin{equation}
\partial_x^\alpha u(x_i, y_j,t) =  
-h_{x}^{-\alpha} \sum_{k = -N_x + i}^{i} g_{k}^{\alpha} u_{i - k, j}(t) + \mathcal{O}(h_x^2)
= \delta_{x}^{\alpha} u_{i j}(t) + \mathcal{O}(h_x^2)
\label{eq2.0a}
\end{equation}
and
\begin{equation}
\partial_y^\beta u(x_i, y_j,t) =  
-h_{y}^{-\beta} \sum_{k = -N_y + j}^{j} g_{k}^{\beta} u_{i, j - k}(t) + \mathcal{O}(h_y^2)
= \delta_{y}^{\beta} u_{i j}(t) + \mathcal{O}(h_y^2),
\label{eq2.0b}
\end{equation}
where 
\begin{equation*}
g_{k}^{\mu} = \frac{(-1)^k \Gamma(1 + \mu)}{\Gamma(\mu/2 - k + 1) \Gamma(\mu/2 + k + 1)}
\quad \mathrm{with} \quad \mu = \alpha, ~\beta \quad \mathrm{and} \quad k \in \mathbb{Z}.
\end{equation*}
With these notations, our semi-discrete scheme is given as
\begin{equation*}
\frac{d u_{i j}(t)}{d t} = (\nu + \mathrm{\mathbf{i}} \eta) \left( \delta_{x}^{\alpha} + \delta_{y}^{\beta} \right) u_{i j}(t)
- (\kappa + \mathrm{\mathbf{i}} \xi) \left| u_{i j}(t) \right|^2 u_{i j}(t) + \gamma u_{i j}(t) .
\end{equation*}
Rewriting it into matrix form, we obtain the following matrix differential equation
\begin{equation}
\dot{U}(t) = A_x U(t) + U(t) A_y - (\kappa + \mathrm{\mathbf{i}} \xi) \left| U(t) \right|^2 U(t) + \gamma U(t), \quad
U(0) = U^0,
\label{eq2.1}
\end{equation}
where $U(t) = \left[ u_{i j}(t) \right]_{\substack{1 \leq i \leq N_x - 1 \\ 1 \leq j \leq N_y - 1}}$,
$U^0 = \left[ u_{0}(x_i,y_j) \right]_{\substack{1 \leq i \leq N_x - 1 \\ 1 \leq j \leq N_y - 1}}$,
$\dot{U}(t) = \left[ \frac{d u_{i j}(t)}{d t} \right]_{\substack{1 \leq i \leq N_x - 1 \\ 1 \leq j \leq N_y - 1}}$
is the first-order derivative of $U(t)$ with respect to $t$.
Here, $A_x$ and $A_y$ are two symmetric Toeplitz matrices \cite{zhang2020exponential} with first columns given by
\begin{equation*}
-\frac{\nu + \mathrm{\mathbf{i}} \eta}{h_{x}^{\alpha}} \left[g_{0}^{\alpha},g_{1}^{\alpha}, \cdots, g_{N_x - 2}^{\alpha} \right]^T
\quad \mathrm{and} \quad
-\frac{\nu + \mathrm{\mathbf{i}} \eta}{h_{y}^{\beta}} \left[g_{0}^{\beta},g_{1}^{\beta}, \cdots, g_{N_y - 2}^{\beta} \right]^T,
\quad \mathrm{respectively}.
\end{equation*}

\subsection{The full-rank Lie-Trotter splitting method}
\label{sec2.2}

The dynamical low-rank approach \cite{koch2007dynamical} can be used to solve Eq.~\eqref{eq2.1} directly,
but it will suffer from the (local) Lipschitz condition of the right-hand side of Eq.~\eqref{eq2.1}.
To overcome this limitation, an alternative way is considered in this work.
More precisely, we first split \eqref{eq2.1} into a stiff linear part and a nonstiff (nonlinear) part.
Then, we apply the dynamical low-rank approximation to solve them separately and combine the solutions with the Lie-Trotter scheme.
This strategy implies that our numerical method only requires a small (local) Lipschitz constant of the nonlinear part, but not of the whole right-hand side of \eqref{eq2.1}.

For a positive integer $M$, let $\tau = \frac{T}{M}$ denote the time step size and $t_k = k \tau~(k = 0, 1, \cdots, M)$.
We split Eq.~\eqref{eq2.1} into the following two subproblems
\begin{equation}
\dot{U}_1 (t) = A_x U_1 (t) + U_1 (t) A_y, \quad U_1(t_0) = U_1^0,
\label{eq2.2}
\end{equation}
and
\begin{equation}
\dot{U}_2 (t) = G(U_2(t)) \triangleq - (\kappa + \mathrm{\mathbf{i}} \xi) \left| U_2 (t) \right|^2 U_2 (t) + \gamma U_2 (t),
\quad U_2(t_0) = U_2^0.
\label{eq2.3}
\end{equation}
Denote by $\Phi_{\tau}^{L}(U_1^0)$ and $\Phi_{\tau}^{G}(U_2^0)$ the solutions of the subproblems \eqref{eq2.2} and \eqref{eq2.3} 
at $t_1$ with initial values $U_1^0$ and $U_2^0$, respectively.
Then, the full-rank Lie-Trotter splitting scheme with time step size $\tau$ is given by
\begin{equation}
\mathcal{L}_{\tau} = \Phi_{\tau}^{L} \circ \Phi_{\tau}^{G}.
\label{eq2.3a}
\end{equation}
Starting with $U_2^0 = U^0$,
the numerical solution $U^1$ of Eq.~\eqref{eq1.1} at $t = t_1$
is thus given by
\begin{equation*}
U^1 = \mathcal{L}_{\tau} (U^0) = \Phi_{\tau}^{L} \circ \Phi_{\tau}^{G} (U^0).
\end{equation*}
Subsequently, the numerical solution of Eq.~\eqref{eq1.1} at $t_k$ is $U^k = \mathcal{L}_{\tau}^{k} (U^0)$.
It is worth noting that the exact solution of \eqref{eq2.2} at $t_1$ can be expressed as
\begin{equation*}
U_{1}(t_{1}) = e^{\tau A_x} U_1^0 e^{\tau A_y}.
\end{equation*}
For large time step sizes $\tau$, this can also be computed efficiently,
e.g. by Taylor interpolation \cite{al2011computing},
the Leja method \cite{caliari2016leja} or Krylov subspace methods \cite{saad1992analysis, lee2010shift}.

The numerical solution $U^1$ is a full rank approximation of $U(t_1)$.
In the next subsection, a low-rank approximation of $U(t)$ is derived.

\subsection{The low-rank approximation}
\label{sec2.3}

Let $\mathcal{M}_{r} = \left\{ X(t) \in \mathbb{C}^{(N_x - 1) \times (N_y - 1)} \mid \mathrm{rank} \left( X(t) \right) = r\right\}$
be the manifold of rank-$r$ matrices.
The aim of this paper is to find a low-rank approximation $X(t) \in \mathcal{M}_{r}$ for the solution of Eq.~\eqref{eq1.1}.
In Section \ref{sec2.2}, we obtained the full rank numerical solution of \eqref{eq1.1}.
Now, we seek after low-rank approximations $X_1(t), X_2(t) \in \mathcal{M}_{r}$ to $U_1(t)$ and $U_2(t)$, respectively.

We start with subproblem \eqref{eq2.2}, which is the easier one.
It can be observed that for any $X \in \mathcal{M}_{r}$,
$A_x X + X A_y \in \mathcal{T}_{X} \mathcal{M}_{r}$, where $\mathcal{T}_{X} \mathcal{M}_{r}$
is the tangent space of $\mathcal{M}_{r}$ at a rank-$r$ matrix $X$.
This implies that \eqref{eq2.2} is rank preserving \cite{helmke1996optimization}.
More precisely, for a given rank-$r$ initial value $X_{1}^{0}$, the solution of
\begin{equation}
\dot{X}_1 (t) = A_x X_1 (t) + X_1 (t) A_y, \quad X_1(t_0) = X_1^0
\label{eq2.4}
\end{equation}
remains rank-$r$ for all $t$.

For the low-rank discretization of subproblem \eqref{eq2.3}, the dynamical low-rank approximation technique \cite{koch2007dynamical} is employed.
The key idea of this technique is to solve the following optimization problem
\begin{equation*}
\min_{X_2(t) \in \mathcal{M}_{r}} \left\| \dot{X}_2 (t) - \dot{U}_2 (t) \right\|, \quad \mathrm{s.t.}~ \dot{X}_2 (t) \in \mathcal{T}_{X_2(t)} \mathcal{M}_{r},
\end{equation*}
where $\mathcal{T}_{X_2(t)} \mathcal{M}_{r}$ is the tangent space of $\mathcal{M}_{r}$ at the current approximation $X_2 (t)$.
Then, the rank-$r$ solution of \eqref{eq2.3} can be obtained by solving the following evolution equation
\begin{equation}
\dot{X}_2 (t) = P(X_2 (t)) G(X_2 (t)), \quad X_2 (t_0) = X_2^0 \in \mathcal{M}_{r},
\label{eq2.5}
\end{equation}
where $P(X_2 (t))$ is the orthogonal projection onto $\mathcal{T}_{X_2(t)} \mathcal{M}_{r}$.
Let the low-rank approximation of $U_2(t)$ at $t_1$ be $X_2^1 = \tilde{\Phi}_{\tau}^{G} (X_2^0)$.
Then, our low-rank Lie-Trotter splitting procedure is given by
\begin{equation}
\mathcal{L}_{\tau, r} = \Phi_{\tau}^{L} \circ \tilde{\Phi}_{\tau}^{G}.
\label{eq2.6}
\end{equation}
Let $X^0$ be a rank-$r$ approximation of the initial value $U^0$. 
We start with $X_{2}^{0} = X^0$
and obtain the rank-$r$ approximation $X^1$ of the solution of \eqref{eq1.1} at $t_1$ as
\begin{equation}
X^1 = \mathcal{L}_{\tau, r} (X^0) = \Phi_{\tau}^{L} \circ \tilde{\Phi}_{\tau}^{G} (X^0).
\label{eq2.7}
\end{equation}
Consequently, the low-rank solution of \eqref{eq1.1} at $t_k$ is $X^k = \mathcal{L}_{\tau, r}^k (X^0)$.

The remaining step is how to solve \eqref{eq2.5}.
We prefer to use the projector-splitting integrator method \cite{lubich2014projector} to solve it numerically.
This method is more robust than the standard numerical integrators (e.g. explicit and implicit Runge-Kutta methods) under over-approximation with a too high rank.
Kieri et al.~\cite{kieri2016discretized} proved that its robustness with respect to small singular values.
This is a crucial property since the rank in most applications is unknown in advance.
In the following subsection, we briefly explain how to use the projector-splitting integrator method to solve Eq.~\eqref{eq2.5}.

\subsection{The projector-splitting integrator}
\label{sec2.4}

According to \cite{lubich2014projector}, every rank-$r$ matrix $X_{2}(t) \in \mathbb{C}^{(N_x - 1) \times (N_y - 1)}$ can be expressed as
$X_{2}(t) = S(t) \Sigma(t) V(t)^{*}$, where $S(t) \in \mathbb{C}^{(N_x - 1) \times r}$ and $V(t) \in \mathbb{C}^{ (N_y - 1) \times r}$ have orthonormal columns,
$\Sigma(t) \in \mathbb{C}^{r \times r}$ is nonsingular and has the same singular values as $X_{2}(t)$, and ${}^{*}$ means conjugate transpose.
This expression is similar to the singular value decomposition (SVD), but $\Sigma(t)$ is not necessarily a diagonal matrix.
Furthermore, the storage requirements and the computation cost are significant reduced if $r$ is very small (i.e., $r \ll \min \{ N_x - 1, N_y - 1 \}$).

From Lemma 4.1 in \cite{nonnenmacher2008dynamical}, the orthogonal projection $P(X_2(t))$ at the current approximation matrix
$X_{2}(t) = S(t) \Sigma(t) V(t)^{*} \in \mathcal{M}_r$ has the following representation
\begin{equation}
\begin{split}
P(X_2(t)) G(X_2(t)) & = S(t) S(t)^{*} G(X_2(t)) - S(t) S(t)^{*} G(X_2(t)) V(t) V(t)^{*}  + G(X_2(t)) V(t) V(t)^{*} \\
& \triangleq P_1 (X_2(t)) G(X_2(t)) - P_1 (X_2(t)) G(X_2(t)) P_2 (X_2(t)) + G(X_2(t)) P_2 (X_2(t)).
\end{split}
\label{eq2.8}
\end{equation}
Here, $P_1 (X_2(t))$ and $P_2 (X_2(t))$ are the orthogonal projections onto the spaces spanned by the range and the corange of $X_2(t)$, respectively.
With this at hand, the low-rank solution of Eq.~\eqref{eq2.5} at $t_1$ can be obtained by solving the evolution equations:
\begin{equation*}
\begin{split}
& \dot{X}_{2}^{I}(t) = P_1 (X_2(t)) G(X_2(t)), \quad X_{2}^{I}(t_0) = X_{2}^{0},\\
& \dot{X}_{2}^{II}(t) = - P_1 (X_2(t)) G(X_2(t)) P_2 (X_2(t)), \quad X_{2}^{II}(t_0) = X_{2}^{I}(t_1),\\
& \dot{X}_{2}^{III}(t) = G(X_2(t)) P_2 (X_2(t)), \quad X_{2}^{III}(t_0) = X_{2}^{II}(t_1).
\end{split}
\end{equation*}
Then, $X_{2}^{III}(t_1)$ is the approximate solution of $X_{2}(t_1)$.
For solving the above three matrix differential equations, a fourth-order Runge-Kutta method is used.
High-order numerical methods for solving \eqref{eq2.6} can be found in \cite{lubich2014projector}.
\section{Convergence analysis}
\label{sec3}

The framework of our low-rank approach has been proposed, but its convergence is still not proved.
This is our next aim.
\subsection{Preliminaries for the convergence analysis}
\label{sec3.1}

In this subsection, some notations and assumptions are introduced for formulating the convergence result.
We consider the Hilbert space $\mathbb{C}^{(N_x - 1) \times (N_y - 1)}$ endowed with the Frobenius norm $\| \cdot \|_{F}$.
Let $X^0$ be a given rank-$r$ approximation of the initial value $U^0$ satisfying
$\left\| X^0 - U^0 \right\|_{F} \leq \sigma$, for some $\sigma \geq 0$. The exact rank-$r$ solution of Eq.~\eqref{eq2.1}
is given as
\begin{equation*}
X(t) = e^{(t - t_0) A_x} X^0 e^{(t - t_0) A_y}
+ \int_{t_0}^{t} e^{(t - s) A_x} P(X (s)) G(X (s)) e^{(t - s) A_y} ds, \quad t_0 \leq t \leq T.
\end{equation*}

The following property and assumption are needed in our convergence analysis.
\begin{assumption}
We assume that
\begin{itemize}
\item[(a)] {$G$ is continuously differentiable in a neighborhood of the exact solution,
and the solution of Eq.~\eqref{eq1.1} is bounded, i.e. $| u(x, y, t) | \leq \delta$,
$(x,y,t) \in \Omega \times (0, T]$, for some $\delta > 0$;}
\item[(b)] {there exists $\varepsilon > 0$ such that
\begin{equation*}
G(X(t)) = \tilde{B}(X(t)) + R(X(t)) \quad \textrm{for}~~t_0 \leq t \leq T,
\end{equation*}
where $\tilde{B}(X(t)) \in \mathcal{T}_{X(t)} \mathcal{M}_{r}$ and $\left\| R(X(t)) \right\|_{F} \leq \varepsilon$.}
\item[(c)] {The exact solution of \eqref{eq1.1} is sufficiently smooth such that the fractional central difference method (Eqs.~\eqref{eq2.0a} and \eqref{eq2.0b}) is second-order accurate; see also \cite{zhang2020exponential,sun2016some}.}
\end{itemize}
\label{assumption1}
\end{assumption}

\begin{property}
\begin{itemize}
\item[(a)] {There exists $C_1 > 0$ such that $A_x$ and $A_y$ satisfy
\begin{equation}
\left\| e^{t A_x} Z e^{t A_y} \right\|_{F} \leq \left\| Z \right\|_{F},
\label{eq3.1}
\end{equation}
\begin{equation}
\left\| e^{t A_x} (A_x Z + Z A_y) e^{t A_y} \right\|_{F} \leq \frac{C_1}{t} \left\| Z \right\|_{F}
\label{eq3.2}
\end{equation}
for all $t > 0$ and $Z \in \mathbb{C}^{(N_x - 1) \times (N_y - 1)}$.}
\item[(b)] {Under Assumption 1(a), the function $G$ is locally Lipschitz continuous and bounded in
a neighborhood of the solution $U(t)$. That is to say, for $\left\| \hat{U} - U(t) \right\|_{F} \leq \tilde{\xi}$,
$\left\| \tilde{U} - U(t) \right\|_{F} \leq \tilde{\xi}$ and $\left\| \bar{U} - U(t) \right\|_{F} \leq \tilde{\xi}$ ($\tilde{\xi} > 0$, $t_0 \leq t \leq T$),
one obtains
\begin{equation}
\left\| G(\hat{U}) - G(\tilde{U}) \right\|_{F} \leq L \left\| \hat{U} - \tilde{U} \right\|_{F}, \quad
\left\| G(\bar{U}) \right\|_{F} \leq H,
\label{eq3.3}
\end{equation}
where the constants $L$ and $H$ depend on $\delta$ and $\tilde{\xi}$.}
\end{itemize}
\label{property1}
\end{property}
\begin{proof}
Let $I_x$ and $I_y$ be the identity matrices with sizes  $N_x - 1$ and $N_y - 1$, respectively.
Let $\otimes$ denote the Kronecker product and $\textrm{vec}(\cdot)$ the columnwise vectorization of a matrix into a column vector. According to \cite[Theorem 3.1]{zhang2020exponential}, for all $t \geq 0$ and all $\bm{z} = \textrm{vec}(Z) \in \mathbb{C}^{(N_x - 1) (N_y - 1)}$,
there exists $C_1 > 0$ such that
\begin{equation*}
\left\| e^{t A} \bm{z} \right\|_2 \leq \| \bm{z} \|_2, \quad
\left\| e^{t A} A \bm{z} \right\|_2 \leq \frac{C_1}{t} \| \bm{z} \|_2,
\end{equation*}
where $A = I_y \otimes A_x + A_y \otimes I_x$ and $\| \cdot \|_2$ is the vector $2$-norm.
These two inequalities can be easily translated to Eqs.~\eqref{eq3.1} and \eqref{eq3.2}, respectively.
It is worth remarking that $C_1$ can be chosen independently of $N_x$ and $N_y$.

For the second property, under Assumption \ref{assumption1}(a), 
it is straightforward to verify that Eq.~\eqref{eq3.3} is true.
Thus, we omit the proof here.
\end{proof}
With these assumptions and properties, our main convergence result is shown in the next subsection.
\subsection{Convergence analysis}
\label{sec3.2}

Following the construction of our approach,
the global error can be split in three terms:
\begin{itemize}
\item[1)] {The global error of the full-rank Lie-Trotter splitting \eqref{eq2.3a}, i.e.
\begin{equation*}
E_{fs}^{k} = \mathcal{U}(t_k) - \mathcal{L}_{\tau}^{k} (U^0).
\end{equation*}
Here, $\mathcal{U}(t_k) = \left[ u(x_i,y_j,t_k) \right]_{1 \leq i,j \leq N - 1}$.}
\item[2)] {The difference between the full-rank initial value $U^0$ and its rank-$r$ approximation $X^0$, both propagated
by the full-rank Lie-Trotter splitting \eqref{eq2.3a}, i.e.
\begin{equation*}
E_{fl}^{k} = \mathcal{L}_{\tau}^{k} (U^0) - \mathcal{L}_{\tau}^{k} (X^0).
\end{equation*}}
\item[3)] {The difference between the full-rank Lie-Trotter splitting \eqref{eq2.3a} and
the low-rank splitting \eqref{eq2.6} applied to $X^0$
\begin{equation*}
E_{lr}^{k} = \mathcal{L}_{\tau}^{k} (X^0) - \mathcal{L}_{\tau,r}^{k} (X^0).
\end{equation*}}
\end{itemize}

Before proving the convergence of our approach, we first estimate $E_{fs}^{k}$ in the following theorem.
\begin{theorem}
Under Assumption \ref{assumption1}, 
for $1 \leq k \leq M$, the error bound
\begin{equation*}
\| E_{fs}^{k} \|_{F} \leq C_2 [\tau (1 + \left| \log \tau \right|) + h_x^2 + h_y^2]
\end{equation*}
holds. Here, the constant $C_2$ depends on $C_1$, $L$ and $H$.
\label{th3.1}
\end{theorem}
\begin{proof}
By the triangle inequality, we first get
\begin{equation*}
\| E_{fs}^{k} \|_{F} \leq \| \mathcal{U}(t_k) - U(t_k)\|_{F} 
+ \| U(t_k) - \mathcal{L}_{\tau}^{k} (U^0) \|_{F}.
\end{equation*}
On the one hand, we know from Assumption \ref{assumption1}(c) and \cite{zhang2020exponential,sun2016some}
that there exists a constant $\bar{C}_2$ satisfying
\begin{equation*}
\| \mathcal{U}(t_k) - U(t_k)\|_{F} \leq \bar{C}_2 \left( h_x^2 + h_y^2 \right).
\end{equation*}
On the other hand, noticing Proposition 1 in \cite{ostermann2019convergence} and Property \ref{property1}.
We have 
\begin{equation*}
\| E_{fs}^{k} \|_{F} \leq \tilde{C}_2 [\tau (1 + \left| \log \tau \right|),
\end{equation*}
where Assumption \ref{assumption1} is used and $\tilde{C}_2$ depends on $C_1$, $L$ and $H$.
Then, the desired result can be obtained by choosing $C_2 = \max \left\{ \bar{C}_2, \tilde{C}_2 \right\}$.
\end{proof}
According to \cite[Section 5]{ostermann2019convergence}, we obtain that
under Assumption \ref{assumption1}, $E_{lr}^{k}$ is bounded on $t_0 \leq t_0 + k \tau \leq T$ as
\begin{equation}
\| E_{lr}^{k} \|_{F} \leq C_3 \varepsilon + C_4 \tau,
\label{eq3.4}
\end{equation}
where Property \ref{property1} is used.
Here, the constants $C_3$ and $C_4$ depend on $H$, $L$ and $T$.

With these two bounds at hand, we now show the boundedness of the global error of the low-rank Lie-Trotter splitting integrator.
\begin{theorem}
Under Assumption \ref{assumption1}, there exists $\tilde{\tau}$ such that for all $0 < \tau \leq \tilde{\tau}$, 
the error of $\mathcal{L}_{\tau, r}$ is bounded on $t_0 \leq t_0 + k \tau \leq T$ by 
\begin{equation}
\left\| \mathcal{U}(t_k) - \mathcal{L}_{\tau, r}^{k} (X^0) \right\|_{F} \leq 
C_3 \varepsilon + C_5 [\tau (1 + \left| \log \tau \right|) + h_x^2 + h_y^2] + e^{L T} \sigma.
\label{eq3.5}
\end{equation}
Here $C_3$ and $C_5$ (containing $C_2$ and $C_4$) are independent of $\tau$ and $k$.
\label{th3.2}
\end{theorem}
\begin{proof}
Due to the stability of $\mathcal{L}_{\tau}$ (see \cite[Proposition 1]{ostermann2019convergence}), 
the error $E_{fl}^{k}$ satisfies
\begin{equation*}
\| E_{fl}^{k} \|_{F} = \| \mathcal{L}_{\tau}^{k} (U^0) - \mathcal{L}_{\tau}^{k} (X^0) \|_{F}
\leq e^{LT} \sigma.
\end{equation*}
	
Combining Theorem \ref{th3.1} and Eq.~\eqref{eq3.4}, the bound \eqref{eq3.5} is obtained. 
\end{proof}


\section{Numerical experiments}
\label{sec5}
In this section, two examples on square domains are provided to verify the convergence rate of our method.
In these examples, we fix $N_x = N_y = N$ and $h_x = h_y = h$.
Let 
\begin{equation*}
\textrm{relerr}(\tau,h) = \frac{\left\| X^M - \mathcal{U}(T) \right\|_{F}}{\left\| \mathcal{U}(T) \right\|_{F}}.
\end{equation*}
Then, we denote
\begin{equation*}
\textrm{rate}_{\tau} = \log_{\tau_1/\tau_2} \frac{\textrm{relerr}(\tau_1,h)}{\textrm{relerr}(\tau_2,h)} 
\quad \mathrm{and} \quad
\textrm{rate}_{h} = \log_{h_1/h_2} \frac{\textrm{relerr}(\tau,h_1)}{\textrm{relerr}(\tau,h_2)}.
\end{equation*}

All experiments are carried out in MATLAB 2018b on a Windows 10 (64 bit) PC with the following configuration:
Intel(R) Core(TM) i7-8700k CPU 3.20 GHz and 16 GB RAM.
\begin{figure}[t]	
	\centering
	\subfigure{
			\includegraphics[width=2.3in,height=2.0in]{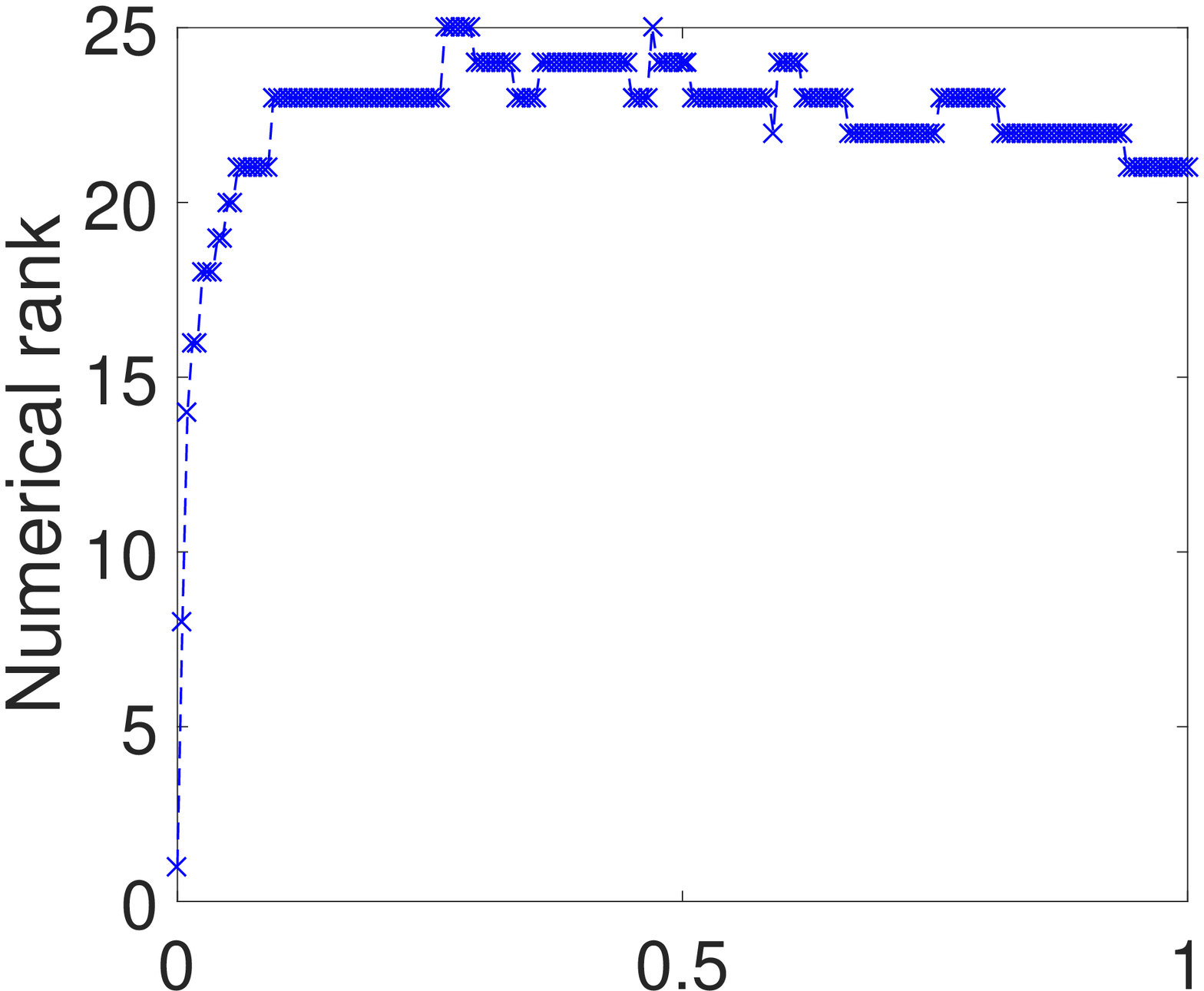}} \hspace{10mm}
\subfigure{
		\includegraphics[width=2.6in,height=2.0in]{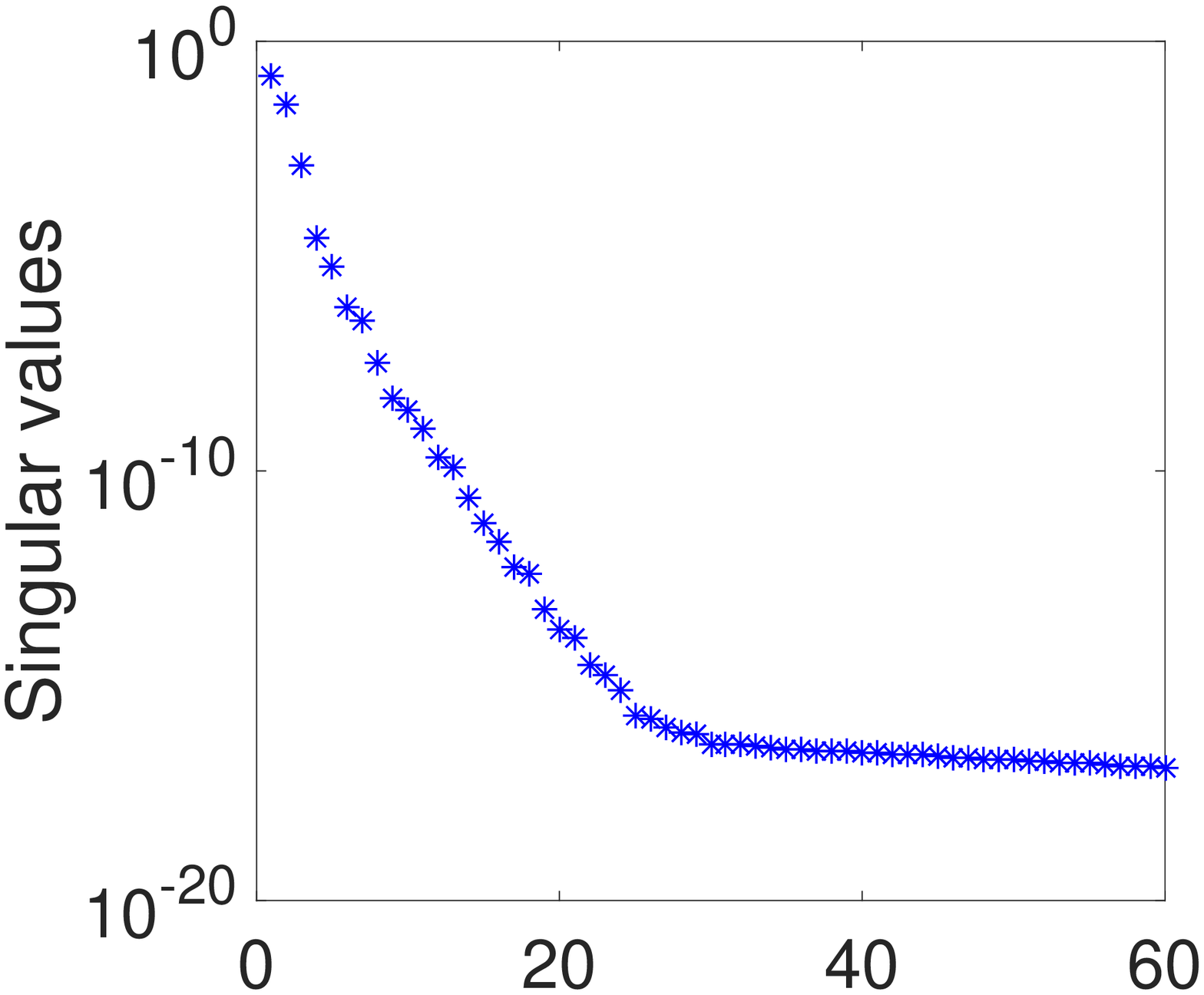}}			
	\caption{Results for Example 1 for $(\alpha, \beta) = (1.5, 1.5)$ and $N = M = 200$.
	Left: Numerical rank of the LBDF2 solution as a function of $t$. 
Right: First $60$ singular values of the LBDF2 solution at $t = T$.}
	\label{fig1}
\end{figure}

\noindent \textbf{Example 1.} In this example, the FGLE \eqref{eq1.1} 
is considered on $\Omega = [-10, 10] \times [-10, 10]$
with $u_0(x,y) = 2\; \mathrm{sech}(x)\; \mathrm{sech}(y) e^{3 \mathrm{\mathbf{i}} (x + y)}$, 
$\nu = \eta = \kappa = \xi = \gamma = 1$ and $T = 1$. 
For this case, the exact solution is unknown.
The numerical solution ($(N, M) = (512, 10000)$) computed by 
the linearized second-order backward differential (LBDF2) scheme \cite{zhang2020linearized}
is used as the reference solution.
Moreover, we denote the numerical solution obtained 
by the LBDF2 scheme \cite{zhang2020linearized} as the LBDF2 solution.

For rank $r \geq 3$, the relative errors in Table \ref{tab1} decrease steadily 
with increasing the number of time steps $M$. 
In this table, the observed temporal convergence order is $1$ as expected.
Table \ref{tab2} reports the relative errors and the observed convergence order in space 
for different values of $\alpha$ and $\beta$. 
It shows that for fixed $M = 10000$, the convergence order in space is indeed $2$ for $r \geq 3$. 
This is in good agreement with our theoretical analysis in Section \ref{sec3.2}.
Fig.~\ref{fig1} (left) shows the rank of the LBDF2 solution with $(\alpha, \beta) = (1.5, 1.5)$ and $N = M = 200$.
It can be observed that the effective rank of the solution is low.
In Fig.~\ref{fig1} (right), we plot the first $60$ singular values of the LBDF2 solution at $T$.
In Fig.~\ref{fig2}, we compare the absolute values of the LBDF2 solution and our low-rank solution (rank $r = 5$) 
at $t = T$ for different values of $\alpha$ and $\beta$.
It can be seen that our proposed method is robust and accurate.
\begin{figure}[t]	
	\centering
	\subfigure{\includegraphics[width=2.1in,height=2.0in]{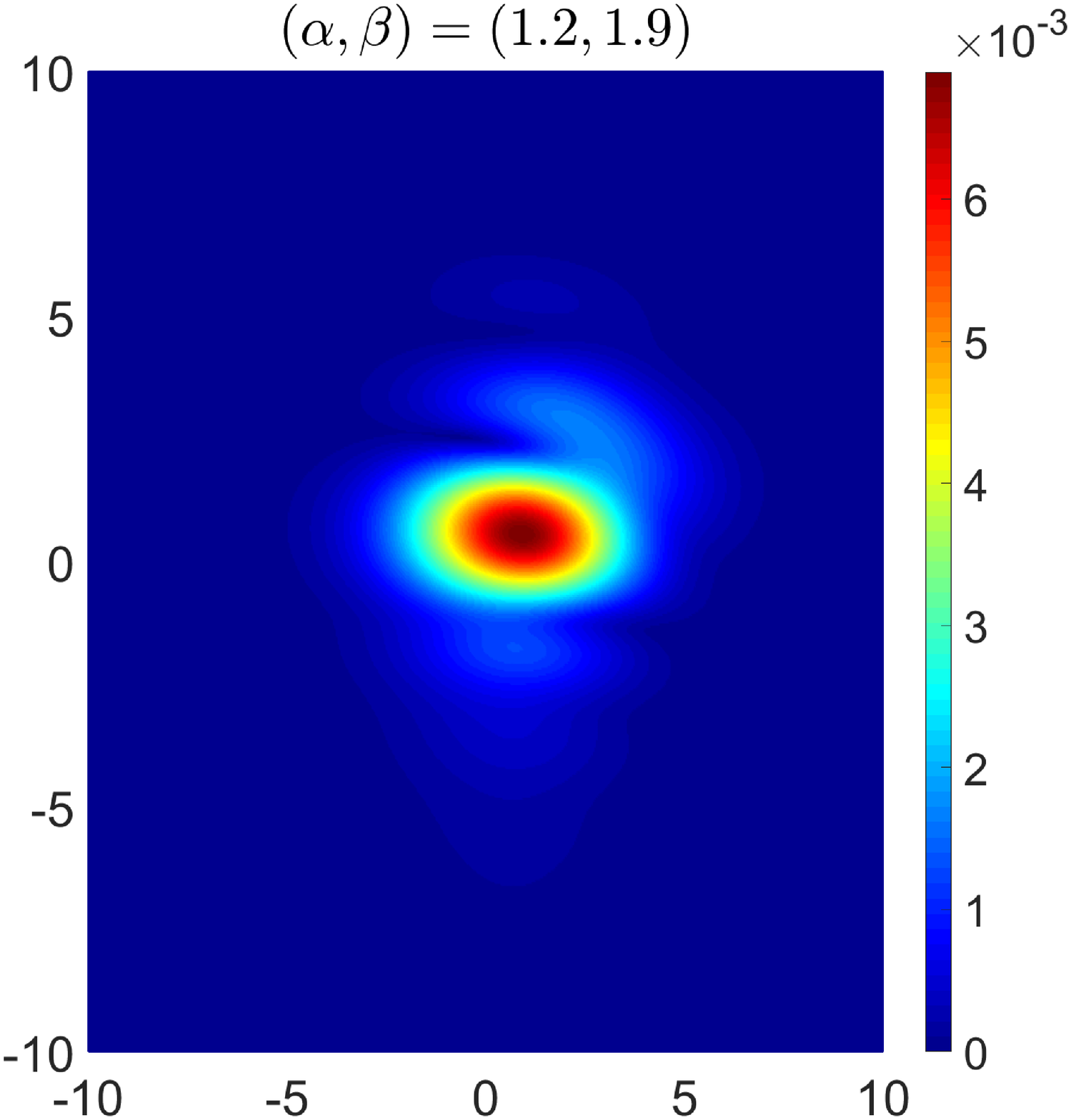}}
	\subfigure{\includegraphics[width=2.1in,height=2.0in]{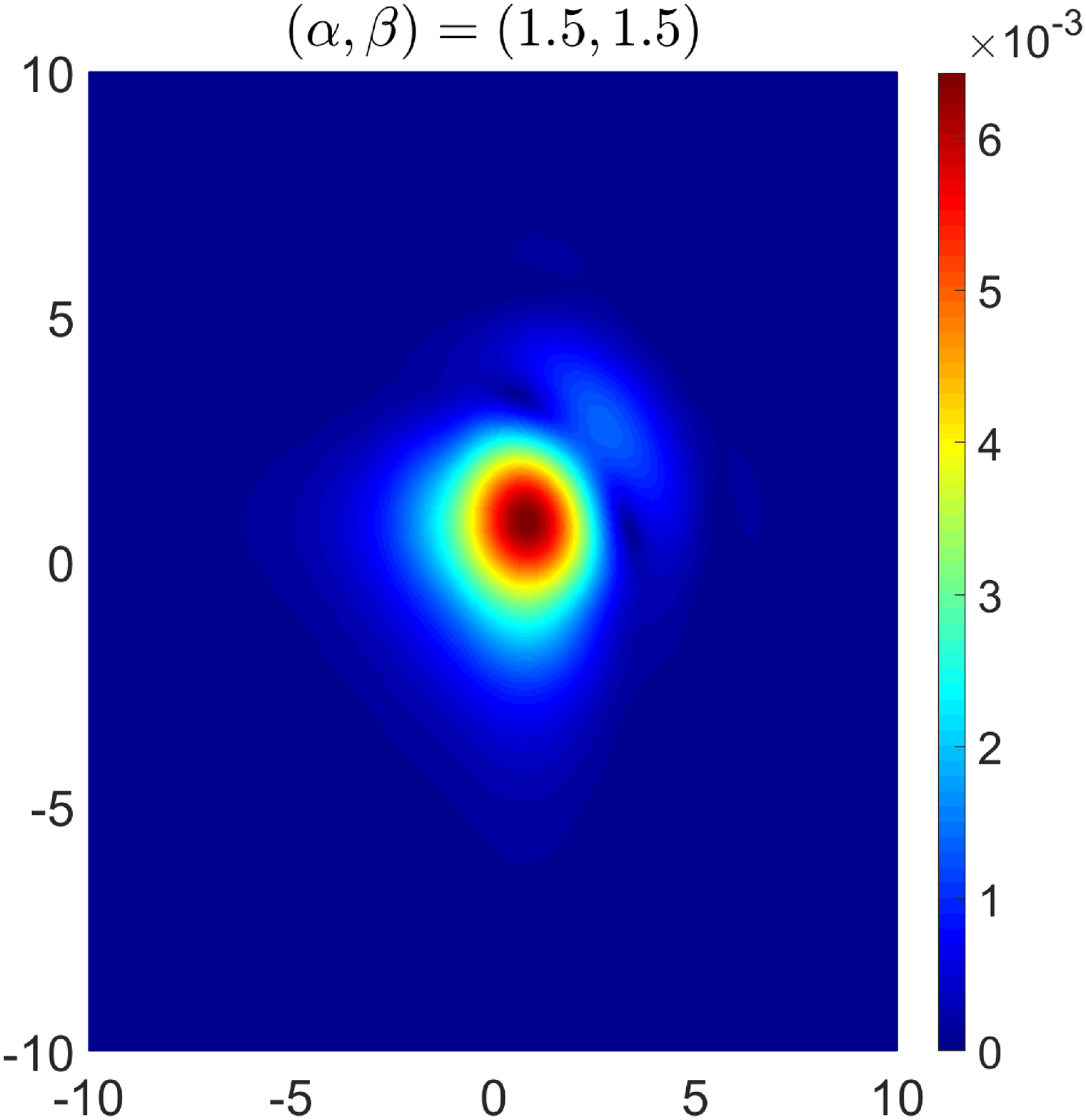}}
	\subfigure{\includegraphics[width=2.1in,height=2.0in]{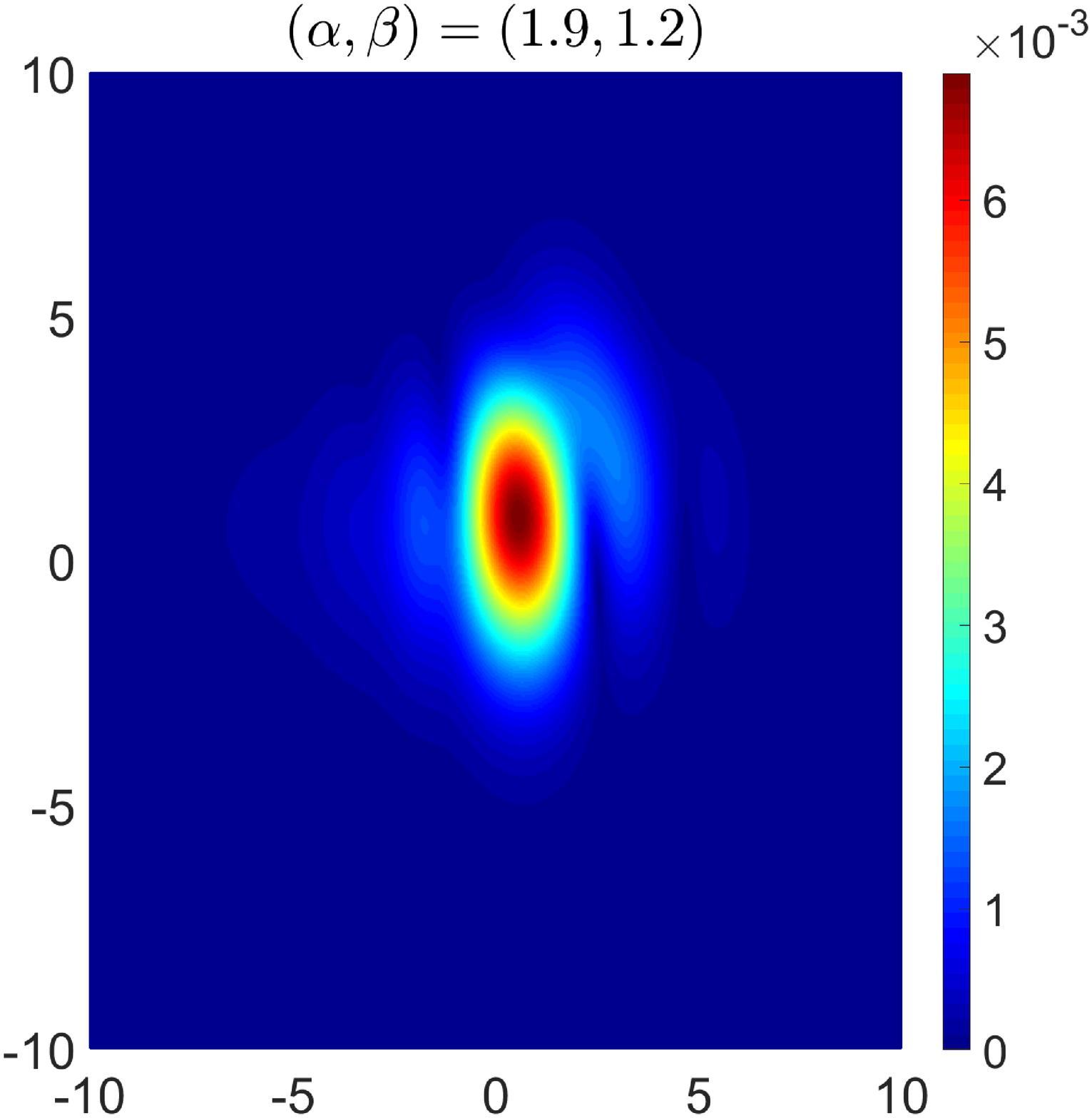}} \\
	\subfigure{\includegraphics[width=2.1in,height=2.0in]{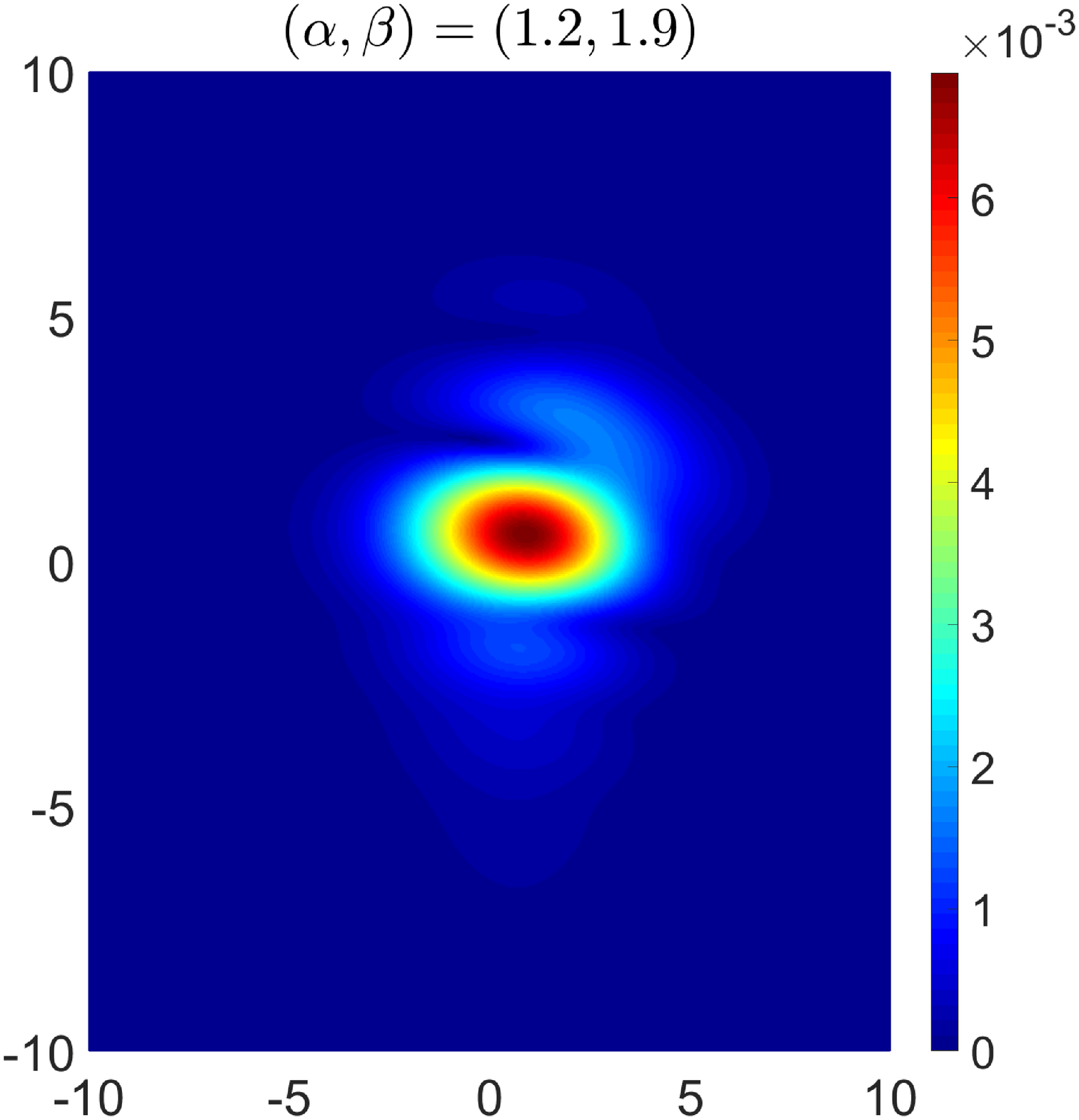}}
\subfigure{\includegraphics[width=2.1in,height=2.0in]{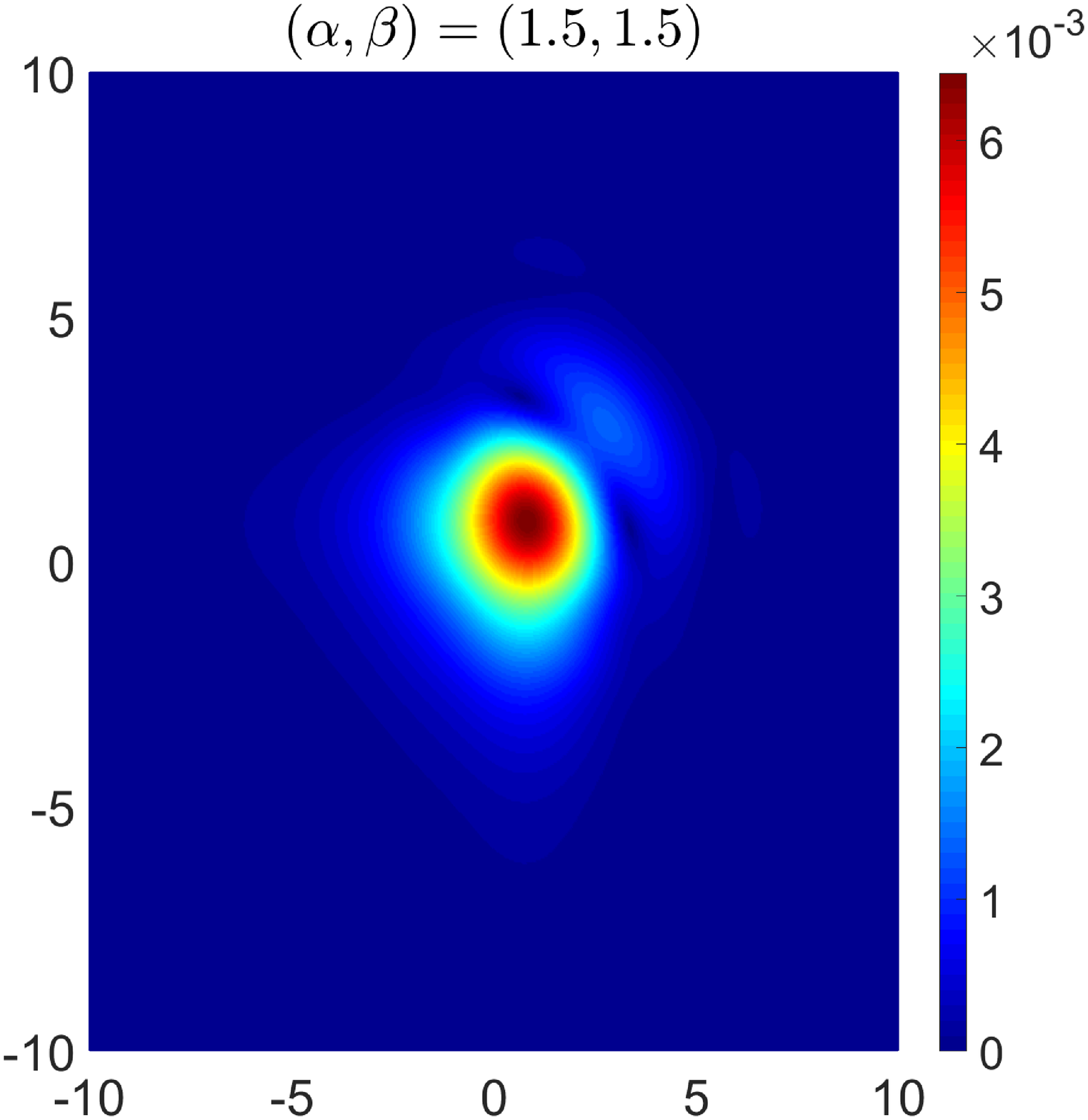}}
\subfigure{\includegraphics[width=2.1in,height=2.0in]{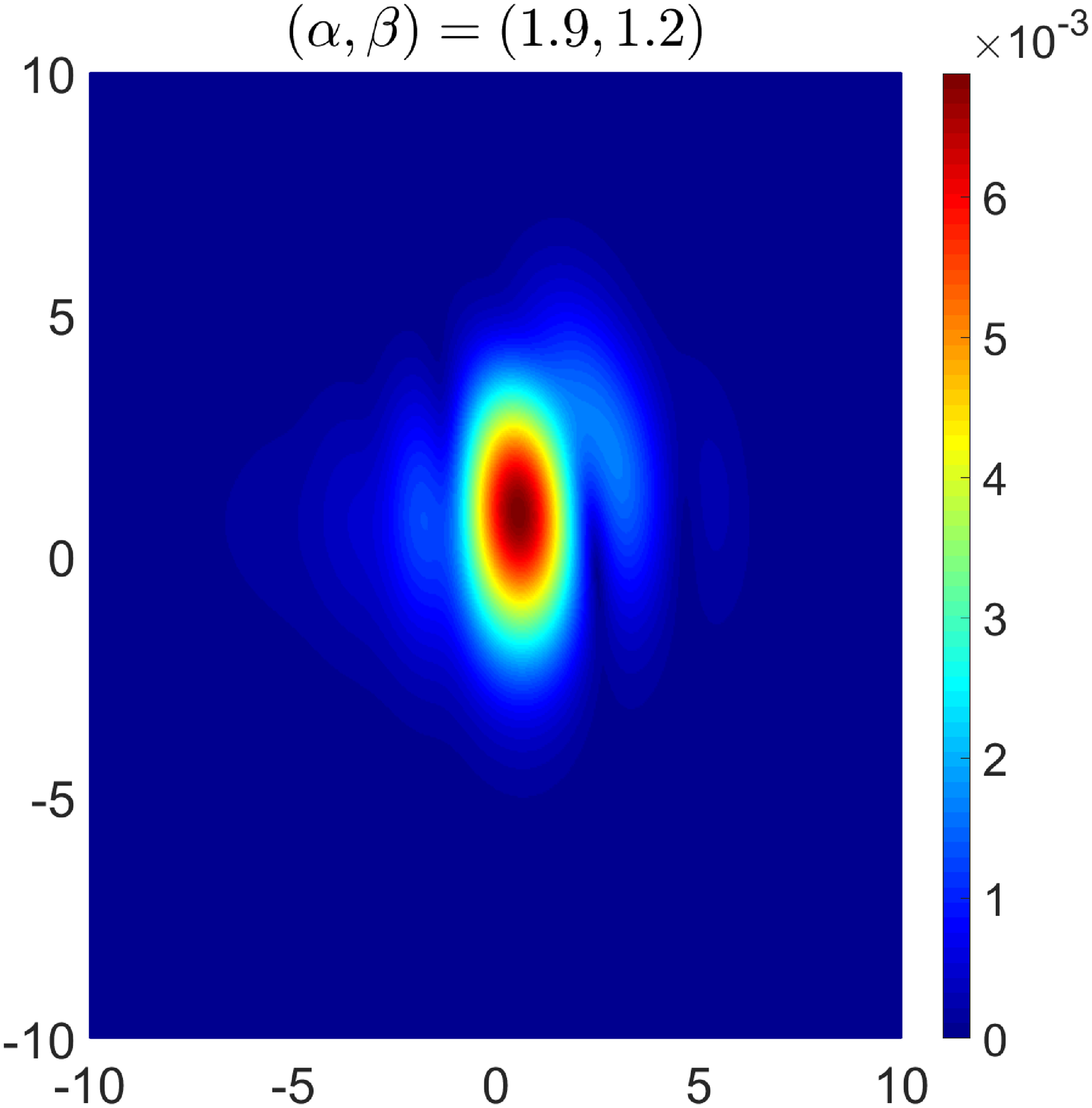}}			
	\caption{Comparison of the absolute values of the LBDF2 solution and our low-rank solution at $t = T$ for $(N, M) = (512, 200)$ 
		and different values of $\alpha$ and $\beta$ for Example 1. 
		Top row: The LBDF2 solution. Bottom row: The low-rank solution (rank $r = 5$).}
	\label{fig2}
\end{figure}
\begin{landscape}
	\begin{table}[t]\scriptsize\tabcolsep=4.0pt
		\begin{center}
			\caption{Relative errors in Frobenius norm and observed temporal convergence orders for $N = 512$ for Example 1.}
			\centering
			\begin{tabular}{cccccccccccc}
				\hline
				& & \multicolumn{2}{c}{$r = 1$} & \multicolumn{2}{c}{$r = 2$} & \multicolumn{2}{c}{$r = 3$} 
				& \multicolumn{2}{c}{$r = 4$} & \multicolumn{2}{c}{$r = 5$} \\
				[-2pt] \cmidrule(lr){3-4} \cmidrule(lr){5-6} \cmidrule(lr){7-8} \cmidrule(lr){9-10} 
				\cmidrule(lr){11-12} \\ [-11pt]
				$(\alpha, \beta)$ & $M$ & $\textrm{relerr}(\tau,h)$ & $\textrm{rate}_\tau$ 
				& $\textrm{relerr}(\tau,h)$ & $\textrm{rate}_\tau$
				& $\textrm{relerr}(\tau,h)$ & $\textrm{rate}_\tau$ & $\textrm{relerr}(\tau,h)$ & $\textrm{rate}_\tau$
				& $\textrm{relerr}(\tau,h)$ & $\textrm{rate}_\tau$ \\
				\hline
				(1.2, 1.9) & 16 & 4.8609E-01 & -- & 8.2522E-01 & -- & 7.9346E-01 & -- & 7.9240E-01 & -- & 7.9253E-01 & -- \\
				& 64 & 5.3468E-01 & -0.0687 & 2.0577E-01 & 1.0019 & 2.0557E-01 & 0.9743 & 2.0455E-01 & 0.9769 & 2.0499E-01 & 0.9755 \\
				& 256 & 5.7709E-01 & -0.0551 & 5.2778E-02 & 0.9815 & 5.1313E-02 & 1.0011 & 5.0386E-02 & 1.0107 & 5.0491E-02 & 1.0107 \\
				& 1024 & 5.8871E-01 & -0.0144 & 2.4422E-02 & 0.5559 & 1.3848E-02 & 0.9448 & 1.2499E-02 & 1.0056 & 1.2581E-02 & 1.0024 \\
				\\
				(1.5, 1.5) & 16 & 5.5049E-01 & -- & 7.9464E-01 & -- & 7.6149E-01 & -- & 7.5911E-01 & -- & 7.5953E-01 & -- \\
				& 64 & 6.4759E-01 & -0.1172 & 1.9941E-01 & 0.9973 & 1.9988E-01 & 0.9648 & 1.9839E-01 & 0.9680 & 	1.9883E-01 & 0.9668 \\ 
				& 256 & 6.9229E-01 & -0.0481 & 5.4117E-02 & 0.9408 & 5.0664E-02 & 0.9901 & 4.9256E-02 & 1.0050 & 4.9399E-02 & 1.0045 \\
				& 1024 & 7.0413E-01 & -0.0122 & 3.0728E-02 & 0.4083 & 1.4516E-02 & 0.9017 & 1.2248E-02 & 1.0039 & 	1.2361E-02 & 0.9993 \\
				\\     
				(1.7, 1.3) & 16 & 5.2921E-01 & -- & 7.8605E-01 & -- & 7.5265E-01 & -- & 7.5042E-01 & -- & 7.5024E-01 & -- \\
				& 64 & 6.0475E-01 & -0.0962 & 1.9632E-01 & 1.0007 & 1.9643E-01 & 0.9690 & 1.9513E-01 & 0.9716 & 	1.9547E-01 & 0.9702 \\
				& 256 & 6.4751E-01 & -0.0493 & 5.2519E-02 & 0.9511 & 4.9500E-02 & 0.9943 & 4.8334E-02 & 1.0067 & 4.8461E-02 & 1.0060 \\
				& 1024 & 6.5899E-01 & -0.0127 & 2.8515E-02 & 0.4406 & 1.3823E-02 & 0.9202 & 1.2008E-02 & 1.0045 & 1.2113E-02 & 1.0001 \\
				\\           
				(1.9, 1.2) & 16 & 4.9746E-01 & -- & 8.3250E-01 & -- & 7.9540E-01 & -- & 7.9353E-01 & -- & 7.9298E-01 & -- \\
				& 64 & 5.3481E-01 & -0.0522 & 2.0620E-01 & 1.0067 & 2.0586E-01 & 0.9750 & 2.0476E-01 & 0.9772 & 2.0499E-01 & 0.9759 \\
				& 256 & 5.7706E-01 & -0.0548 & 5.2760E-02 & 0.9833 & 5.1371E-02 & 1.0013 & 5.0393E-02 & 1.0113 & 5.0492E-02 & 1.0107 \\
				& 1024 & 5.8870E-01 & -0.0144 & 2.4418E-02 & 0.5557 & 1.3867E-02 & 0.9446 & 1.2499E-02 & 1.0057 & 1.2581E-02 & 1.0024 \\
				\hline
			\end{tabular}
			\label{tab1}
		\end{center}
	\end{table}
	\begin{table}[t]\scriptsize\tabcolsep=4.0pt
		\begin{center}
			\caption{Relative errors in Frobenius norm and observed spatial convergence orders 
				for $M = 10000$ for Example 1.}
			\centering
			\begin{tabular}{cccccccccccc}
				\hline
				& & \multicolumn{2}{c}{$r = 1$} & \multicolumn{2}{c}{$r = 2$} & \multicolumn{2}{c}{$r = 3$} 
				& \multicolumn{2}{c}{$r = 4$} & \multicolumn{2}{c}{$r = 5$} \\
				[-2pt] \cmidrule(lr){3-4} \cmidrule(lr){5-6} \cmidrule(lr){7-8} \cmidrule(lr){9-10} 
				\cmidrule(lr){11-12} \\ [-11pt]
				$(\alpha, \beta)$ & $N$ & $\textrm{relerr}(\tau,h)$ & $\textrm{rate}_h$ 
				& $\textrm{relerr}(\tau,h)$ & $\textrm{rate}_h$
				& $\textrm{relerr}(\tau,h)$ & $\textrm{rate}_h$ & $\textrm{relerr}(\tau,h)$ & $\textrm{rate}_h$
				& $\textrm{relerr}(\tau,h)$ & $\textrm{rate}_h$ \\
				\hline
				(1.2, 1.9) & 32 & 7.1496E-01 & -- & 6.0075E-01 & -- & 6.0033E-01 & -- & 6.0079E-01 & -- & 6.0069E-01 & -- \\
				& 64 & 5.8194E-01 & 0.2970 & 1.2508E-01 & 2.2639 & 1.2512E-01 & 2.2624 & 1.2478E-01 & 2.2675 & 1.2482E-01 & 2.2668 \\
				& 128 & 5.8821E-01 & -0.0155 & 3.5052E-02 & 1.8353 & 2.9685E-02 & 2.0755 & 2.8836E-02 & 2.1134 & 2.8893E-02 & 2.1111 \\
				& 256 & 5.9137E-01 & -0.0077 & 2.2788E-02 & 0.6212 & 7.9685E-03 & 1.8974 & 6.1959E-03 & 2.2185 & 6.2558E-03 & 2.2075 \\
				\\
				(1.5, 1.5) & 32 & 8.2826E-01 & -- & 6.2358E-01 & -- & 6.2338E-01 & -- & 6.2391E-01 & -- & 6.2380E-01 & -- \\
				& 64 & 6.9735E-01 & 0.2482 & 1.2391E-01 & 2.3313 & 1.2393E-01 & 2.3306 & 1.2330E-01 & 2.3392 & 1.2337E-01 & 2.3381 \\
				& 128 & 7.0396E-01 & -0.0136 & 3.8797E-02 & 1.6753 & 2.9802E-02 & 2.0560 & 2.8319E-02 & 2.1223 & 2.8406E-02 & 2.1187 \\
				& 256 & 7.0691E-01 & -0.0060 & 2.9772E-02 & 0.3820 & 9.2925E-03 & 1.6813 & 6.1664E-03 & 2.1993 & 6.2448E-03 & 2.1855 \\
				\\
				(1.7, 1.3) & 32 & 7.9484E-01 & -- & 6.2588E-01 & -- & 6.2572E-01 & -- & 6.2618E-01 & -- & 6.2608E-01 & -- \\
				& 64 & 6.5274E-01 & 0.2842 & 1.2632E-01 & 2.3088 & 1.2627E-01 & 2.3090 & 1.2579E-01 & 2.3156 & 1.2585E-01 & 2.3146 \\
				& 128 & 6.5868E-01 & -0.0131 & 3.7824E-02 & 1.7397 & 3.0045E-02 & 2.0713 & 2.8901E-02 & 2.1218 & 2.8976E-02 & 2.1188 \\ 
				& 256 & 6.6165E-01 & -0.0065 & 2.7462E-02 & 0.4619 & 8.6706E-03 & 1.7929 & 6.2136E-03 & 2.2176 & 6.2862E-03 & 2.2046 \\
				\\
				(1.9, 1.2) & 32 & 7.1495E-01 & -- & 6.0075E-01 & -- & 6.0033E-01 & -- & 6.0079E-01 & -- & 6.0069E-01 & -- \\
				& 64 & 5.8194E-01 & 0.2970 & 1.2507E-01 & 2.2640 & 1.2512E-01 & 2.2624 & 1.2478E-01 & 2.2675 & 1.2482E-01 & 2.2668 \\
				& 128 & 5.8821E-01 & -0.0155 & 3.5052E-02 & 1.8352 & 2.9686E-02 & 2.0755 & 2.8836E-02 & 2.1134 & 2.8893E-02 & 2.1111 \\
				& 256 & 5.9137E-01 & -0.0077 & 2.2788E-02 & 0.6212 & 7.9705E-03 & 1.8970 & 6.1959E-03 & 2.2185 & 6.2558E-03 & 2.2075 \\
				\hline
			\end{tabular}
			\label{tab2}
		\end{center}
	\end{table}
\end{landscape}
\begin{figure}[H]	
	\centering
	\subfigure{\includegraphics[width=2.1in,height=2.0in]{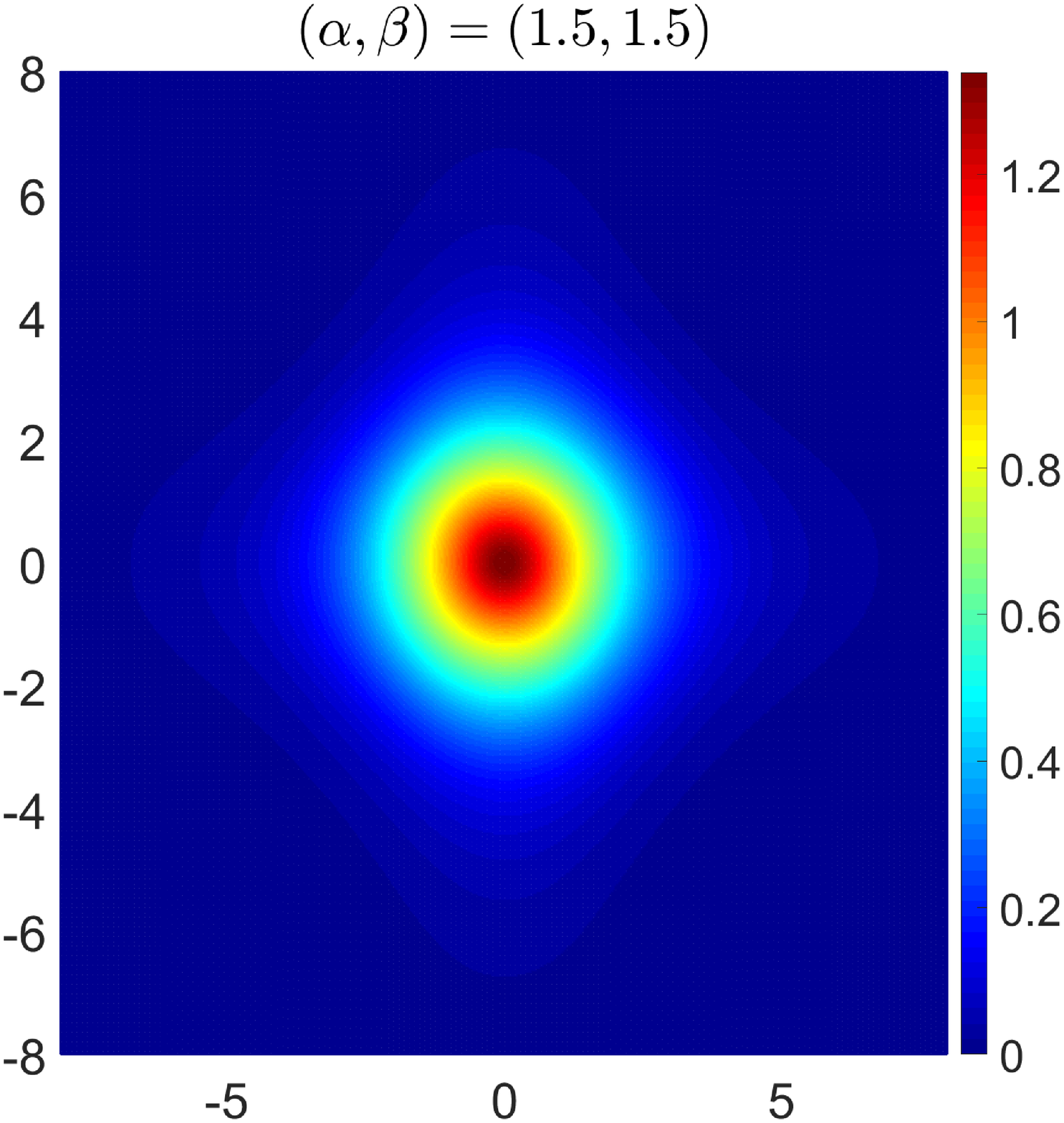}}
	\subfigure{\includegraphics[width=2.1in,height=2.0in]{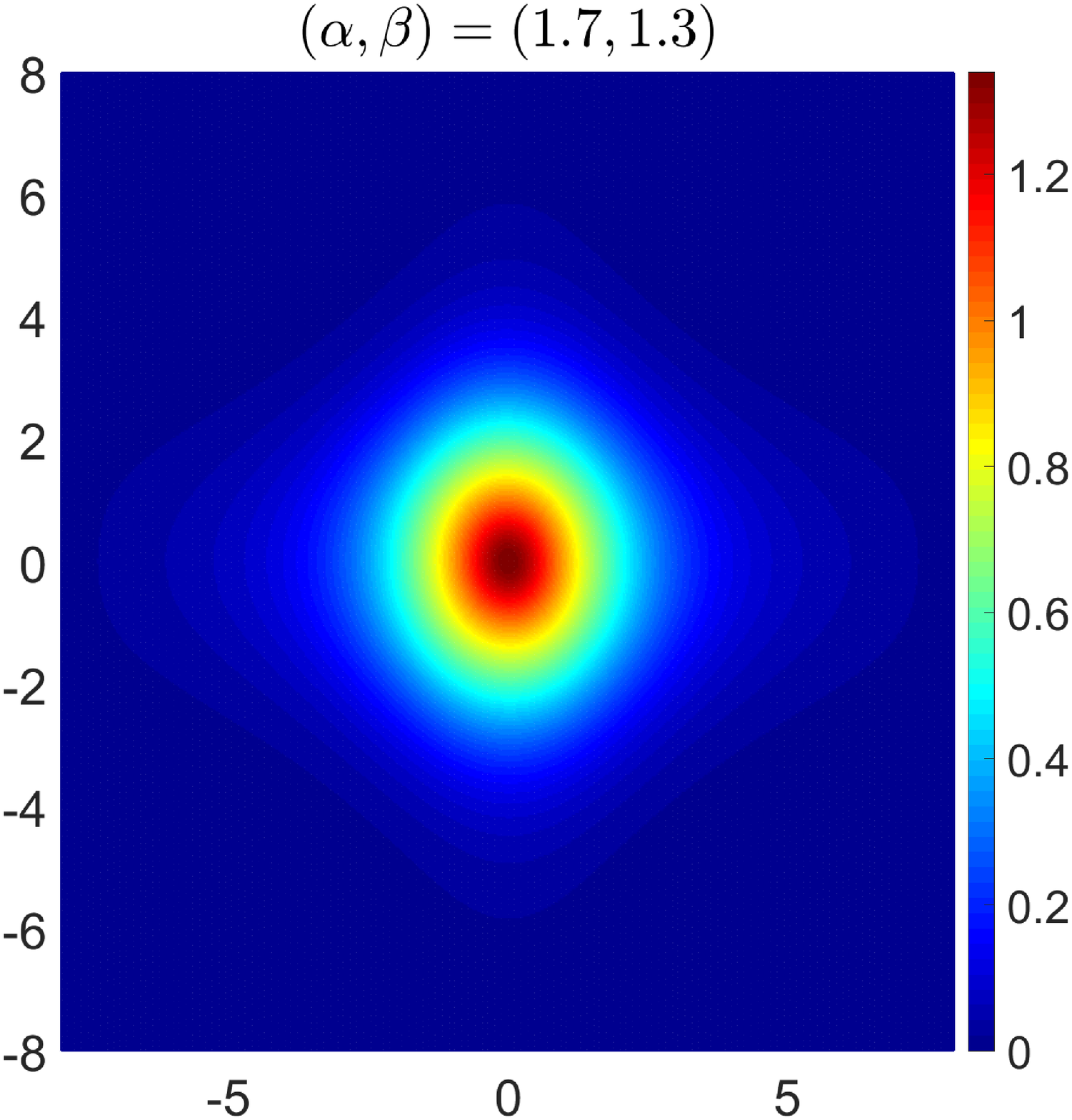}}
	\subfigure{\includegraphics[width=2.1in,height=2.0in]{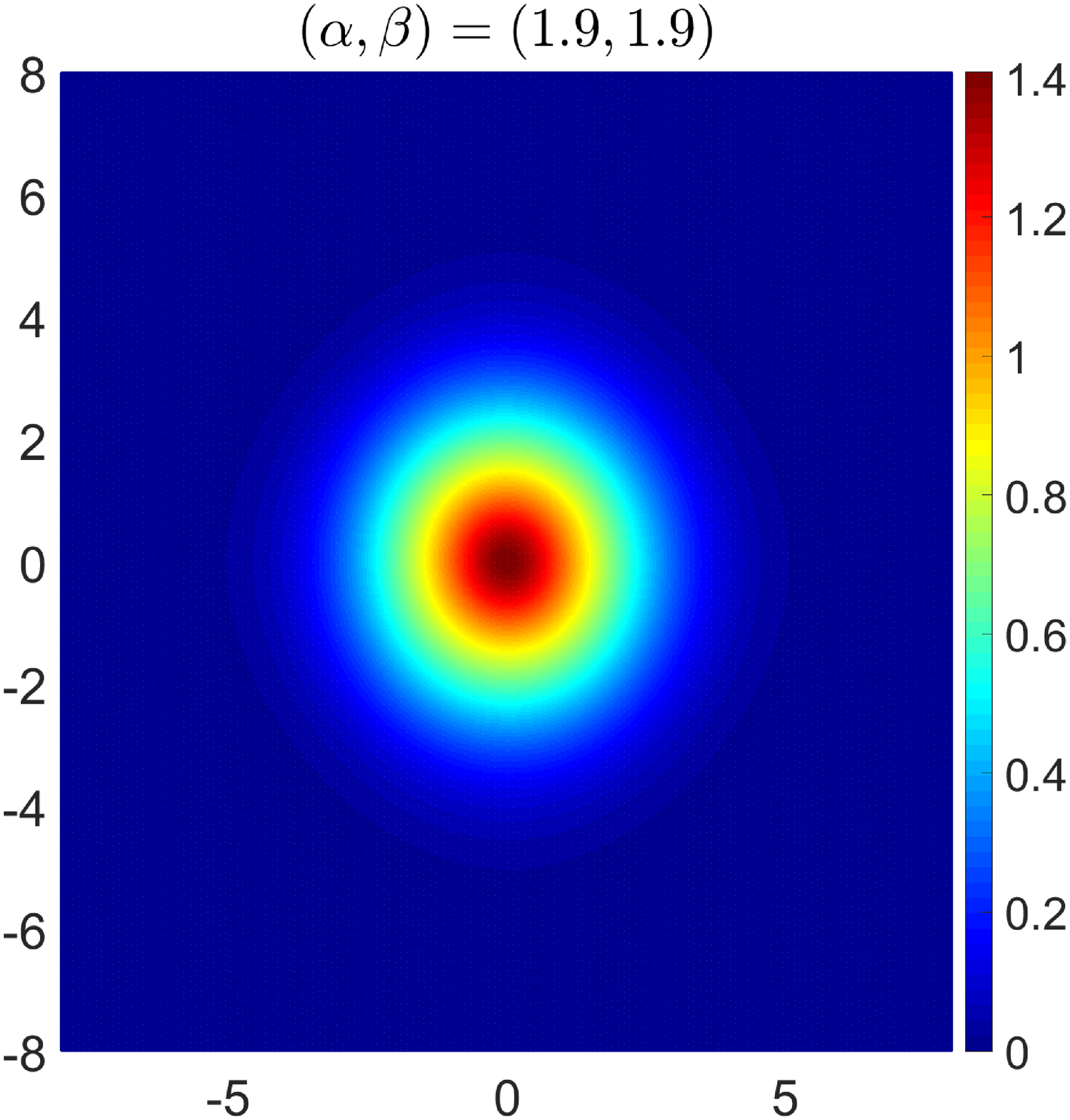}} \\
	\subfigure{\includegraphics[width=2.1in,height=2.0in]{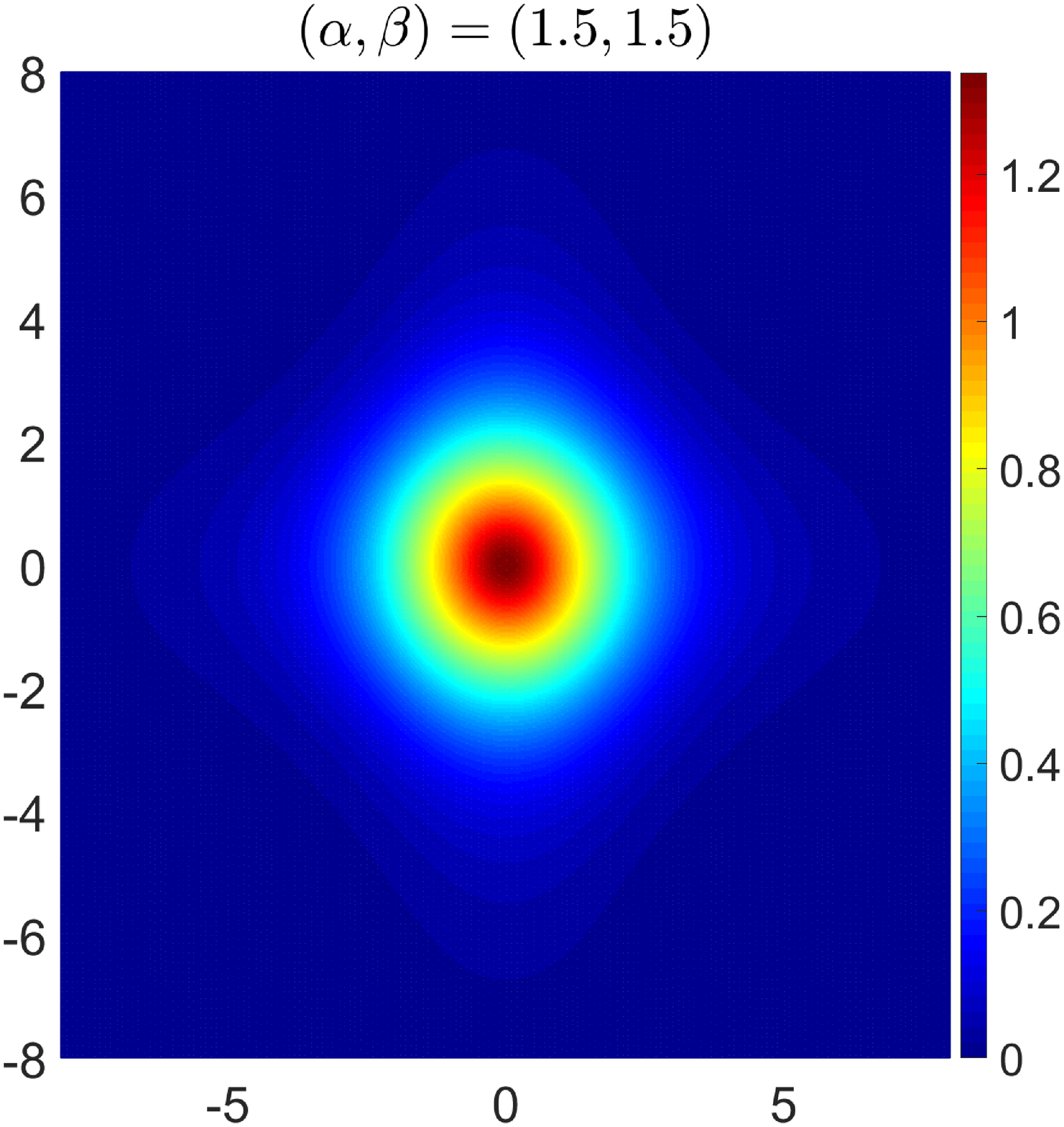}}
	\subfigure{\includegraphics[width=2.1in,height=2.0in]{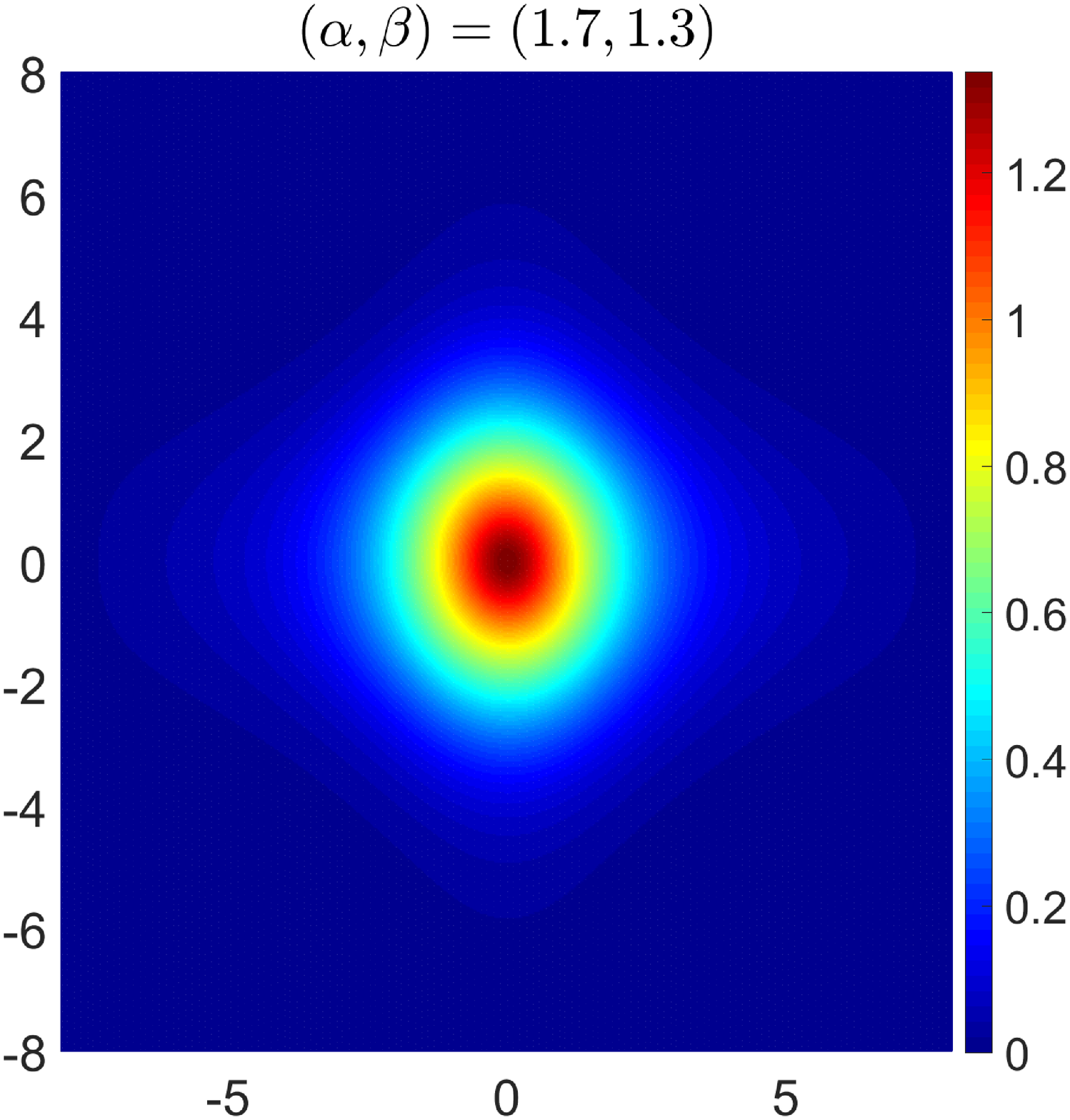}}
	\subfigure{\includegraphics[width=2.1in,height=2.0in]{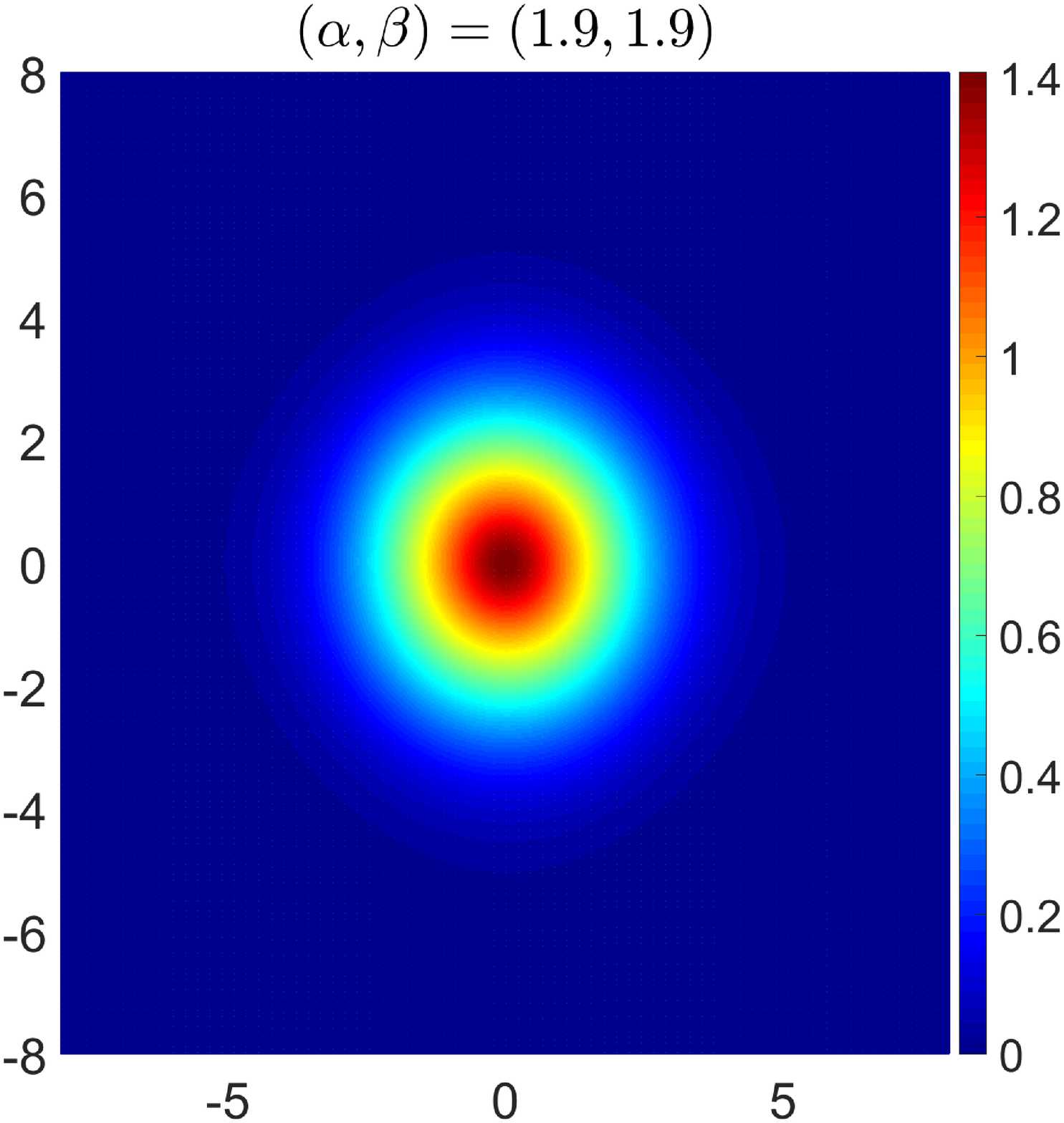}}			
	\caption{Comparison of the absolute values of the LBDF2 solution and our low-rank solution at $t = T$ for $(N, M) = (512, 200)$ 
		and different values of $\alpha$ and $\beta$ for Example 2. 
		Top row: The LBDF2 solution. Bottom row: The low-rank solution (rank $r = 8$).}
	\label{fig3}
\end{figure}

\noindent \textbf{Example 2.} In this example, in Eq.~\eqref{eq1.1}, 
we choose $\nu = \kappa = 1$, $\eta = 0.5$, $\xi = -5$, $\gamma = 3$,
$\Omega = [-8, 8] \times [-8, 8]$, $T = 1$ and the initial value 
$u_0(x,y) = e^{-2 \left( x ^2 + y^2 \right)} e^{\mathrm{\mathbf{i}} S_0}$,
where $S_0 = \frac{1}{e^{x + y} + e^{-(x + y)}}$.
The exact solution in this case is unknown. Similar to Example 1, we choose the LBDF2 solution with $(N, M) = (512, 10000)$ as the reference solution.

Table \ref{tab3} lists the errors and the observed temporal convergence orders.
From this table, we can see that for fixed $N = 512$ and $r \geq 6$, 
the observed temporal convergence order is $1$. 
In Table \ref{tab4}, for $r = 8$, the observed spatial convergence order is $2$.
In a word, these numerical results are consistent with the theoretical result in Section \ref{sec3.2}.
Fig.~\ref{fig3} compares the absolute values of the LBDF2 solution and the low-rank solution ($r = 8$) 
for different values of $\alpha$ and $\beta$.
Moreover, if the rank $r \geq 8$, the errors are only slightly different 
and the convergence order will not improve significantly.

\begin{landscape}
	\begin{table}[t]\scriptsize\tabcolsep=4.0pt
		\begin{center}
			\caption{Relative errors in Frobenius norm and observed temporal convergence orders for $N = 512$ for Example 2.}
			\centering
			\begin{tabular}{cccccccccccc}
				\hline
				& & \multicolumn{2}{c}{$r = 1$} & \multicolumn{2}{c}{$r = 2$} & \multicolumn{2}{c}{$r = 4$} 
				& \multicolumn{2}{c}{$r = 6$} & \multicolumn{2}{c}{$r = 8$} \\
				[-2pt] \cmidrule(lr){3-4} \cmidrule(lr){5-6} \cmidrule(lr){7-8} \cmidrule(lr){9-10} 
				\cmidrule(lr){11-12} \\ [-11pt]
				$(\alpha, \beta)$ & $M$ & $\textrm{relerr}(\tau,h)$ & $\textrm{rate}_\tau$ 
				& $\textrm{relerr}(\tau,h)$ & $\textrm{rate}_\tau$
				& $\textrm{relerr}(\tau,h)$ & $\textrm{rate}_\tau$ & $\textrm{relerr}(\tau,h)$ & $\textrm{rate}_\tau$
				& $\textrm{relerr}(\tau,h)$ & $\textrm{rate}_\tau$ \\
				\hline
				(1.2, 1.9) & 16 & 2.0955E-01 & -- & 2.0940E-01 & -- & 8.9892E-02 & -- & 8.9162E-02 & -- & 8.9135E-02 & -- \\
				& 64 & 1.9990E-01 & 0.0340 & 1.6160E-01 & 0.1869 & 2.5814E-02 & 0.9000 & 2.1950E-02 & 1.0111 & 2.1913E-02 & 1.0121 \\
				& 256 & 2.0063E-01 & -0.0026 & 1.5059E-01 & 0.0509 & 1.4928E-02 & 0.3951 & 5.5422E-03 & 0.9928 & 5.4473E-03 & 1.0041 \\
				& 1024 & 2.0100E-01 & -0.0013 & 1.4999E-01 & 0.0029 & 1.4018E-02 & 0.0454 & 1.6603E-03 & 0.8695 & 1.3588E-03 & 1.0016 \\
				\\
				(1.5, 1.5) & 16 & 2.0938E-01 & -- & 2.0924E-01 & -- & 8.8782E-02 & -- & 8.8147E-02 & -- & 8.8130E-02 & -- \\
				& 64 & 2.0016E-01 & 0.0325 & 1.7916E-01 & 0.1120 & 2.5573E-02 & 0.8978 & 2.1698E-02 & 1.0112 & 2.1669E-02 & 1.0120 \\
				& 256 & 2.0096E-01 & -0.0029 & 1.5866E-01 & 0.0877 & 1.4931E-02 & 0.3882 & 5.4758E-03 & 0.9932 & 5.3878E-03 & 1.0039 \\
				& 1024 & 2.0134E-01 & -0.0014 & 1.5945E-01 & -0.0036 & 1.4059E-02 & 0.0434 & 1.6403E-03 & 0.8696 & 1.3444E-03 & 1.0014 \\
				\\
				(1.7, 1.3) & 16 & 2.1149E-01 & -- & 2.1135E-01 & -- & 8.9162E-02 & -- & 8.8495E-02 & -- & 8.8470E-02 & -- \\
				& 64 & 2.0198E-01 & 0.0332 & 1.6305E-01 & 0.1872 & 2.5737E-02 & 0.8963 & 2.1780E-02 & 1.0113 & 2.1745E-02 & 1.0123 \\
				& 256 & 2.0270E-01 & -0.0026 & 1.5353E-01 & 0.0434 & 1.5076E-02 & 0.3858 & 5.4999E-03 & 0.9928 & 5.4059E-03 & 1.0040 \\
				& 1024 & 2.0306E-01 & -0.0013 & 1.5552E-01 & -0.0093 & 1.4197E-02 & 0.0433 & 1.6549E-03 & 0.8663 & 1.3484E-03 & 1.0016 \\
				\\
				(1.9, 1.2) & 16 & 2.0979E-01 & -- & 2.0964E-01 & -- & 8.9854E-02 & -- & 8.9167E-02 & -- & 8.9134E-02 & -- \\
				& 64 & 1.9999E-01 & 0.0345 & 1.5760E-01 & 0.2058 & 2.5825E-02 & 0.8994 & 2.1951E-02 & 1.0111 & 2.1913E-02 & 1.0121 \\
				& 256 & 2.0066E-01 & -0.0024 & 1.4966E-01 & 0.0373 & 1.4938E-02 & 0.3949 & 5.5426E-03 & 0.9928 & 5.4473E-03 & 1.0041 \\
				& 1024 & 2.0101E-01 & -0.0013 & 1.5024E-01 & -0.0028 & 1.4021E-02 & 0.0457 & 1.6605E-03 & 0.8695 & 1.3588E-03 & 1.0016 \\
				\hline
			\end{tabular}
			\label{tab3}
		\end{center}
	\end{table}
	\begin{table}[t]\scriptsize\tabcolsep=4.0pt
		\begin{center}
			\caption{Relative errors in Frobenius norm and observed spatial convergence orders 
				for $M = 10000$ for Example 2.}
			\centering
			\begin{tabular}{cccccccccccc}
				\hline
				& & \multicolumn{2}{c}{$r = 1$} & \multicolumn{2}{c}{$r = 2$} & \multicolumn{2}{c}{$r = 4$} 
				& \multicolumn{2}{c}{$r = 6$} & \multicolumn{2}{c}{$r = 8$} \\
				[-2pt] \cmidrule(lr){3-4} \cmidrule(lr){5-6} \cmidrule(lr){7-8} \cmidrule(lr){9-10} 
				\cmidrule(lr){11-12} \\ [-11pt]
				$(\alpha, \beta)$ & $N$ & $\textrm{relerr}(\tau,h)$ & $\textrm{rate}_h$ 
				& $\textrm{relerr}(\tau,h)$ & $\textrm{rate}_h$
				& $\textrm{relerr}(\tau,h)$ & $\textrm{rate}_h$ & $\textrm{relerr}(\tau,h)$ & $\textrm{rate}_h$
				& $\textrm{relerr}(\tau,h)$ & $\textrm{rate}_h$ \\
				\hline
				(1.2, 1.9) & 32 & 2.2126E-01 & -- & 1.8457E-01 & -- & 3.7148E-02 & -- & 3.4214E-02 & -- & 3.4212E-02 & -- \\
				& 64 & 2.0525E-01 & 0.1084 & 1.6523E-01 & 0.1597 & 1.6935E-02 & 1.1333 & 9.3369E-03 & 1.8736 & 9.2804E-03 & 1.8822 \\
				& 128 & 2.0204E-01 & 0.0227 & 1.6095E-01 & 0.0379 & 1.4196E-02 & 0.2545 & 2.4582E-03 & 1.9253 & 2.2641E-03 & 2.0352 \\
				& 256 & 2.0130E-01 & 0.0053 & 1.5971E-01 & 0.0112 & 1.3986E-02 & 0.0215 & 1.0737E-03 & 1.1950 & 5.2148E-04 & 2.1183 \\
				\\
				(1.5, 1.5) & 32 & 2.2202E-01 & -- & 1.8081E-01 & -- & 3.5474E-02 & -- & 3.2514E-02 & -- & 3.2513E-02 & -- \\
				& 64 & 2.0562E-01 & 0.1107 & 1.5709E-01 & 0.2029 & 1.6569E-02 & 1.0983 & 8.6606E-03 & 1.9085 & 8.6039E-03 & 1.9180 \\
				& 128 & 2.0239E-01 & 0.0228 & 1.6308E-01 & -0.0540 & 1.4206E-02 & 0.2220 & 2.3027E-03 & 1.9111 & 2.1004E-03 & 2.0343 \\
				& 256 & 2.0164E-01 & 0.0054 & 1.5177E-01 & 0.1037 & 1.4030E-02 & 0.0180 & 1.0497E-03 & 1.1333 & 4.8398E-04 & 2.1176 \\
				\\
				(1.7, 1.3) & 32 & 2.2395E-01 & -- & 1.9556E-01 & -- & 3.6112E-02 & -- & 3.3115E-02 & -- & 3.3112E-02 & -- \\
				& 64 & 2.0741E-01 & 0.1107 & 1.6221E-01 & 0.2697 & 1.6822E-02 & 1.1021 & 8.8906E-03 & 1.8971 & 8.8319E-03 & 1.9066 \\
				& 128 & 2.0412E-01 & 0.0231 & 1.6714E-01 & -0.0432 & 1.4352E-02 & 0.2291 & 2.3603E-03 & 1.9133 & 2.1549E-03 & 	2.0351 \\
				& 256 & 2.0337E-01 & 0.0053 & 1.5975E-01 & 0.0652 & 1.4167E-02 & 0.0187 & 1.0696E-03 & 1.1419 & 4.9634E-04 & 2.1182 \\
				\\
				(1.9, 1.2) & 32 & 2.2126E-01 & -- & 1.8671E-01 & -- & 3.7143E-02 & -- & 3.4214E-02 & -- & 3.4212E-02 & -- \\
				& 64 & 2.0525E-01 & 0.1084 & 1.7556E-01 & 0.0888 & 1.6936E-02 & 1.1330 & 9.3369E-03 & 1.8736 & 9.2804E-03 & 1.8822 \\
				& 128 & 2.0204E-01 & 0.0227 & 1.5858E-01 & 0.1468 & 1.4196E-02 & 0.2546 & 2.4582E-03 & 1.9253 & 2.2641E-03 & 2.0352 \\
				& 256 & 2.0130E-01 & 0.0053 & 1.6113E-01 & -0.0230 & 1.3986E-02 & 0.0215 & 1.0737E-03 & 1.1950 & 5.2148E-04 & 2.1183 \\
				\hline
			\end{tabular}
			\label{tab4}
		\end{center}
	\end{table}
\end{landscape}

\section{Concluding remarks}
\label{sec6}
In this paper, we propose a numerical integration method based on a dynamical low-rank approximation
for solving the 2D space fractional Ginzburg-Landau equations \eqref{eq1.1}.
First, we use a second-order difference method to approximate the two Riesz fractional derivatives.
From this, the matrix differential equation \eqref{eq2.1} is obtained.  
Next, we split Eq.~\eqref{eq2.1} into a stiff linear part and a nonstiff (nonlinear) part (see Eqs.~\eqref{eq2.2} and \eqref{eq2.3}),
and then apply a dynamical low-rank approach.
The convergence of our proposed method is studied in Section \ref{sec3}. 
Numerical examples strongly support the theoretical result which is given in Theorem \ref{th3.2}.
Based on this work, three future research directions are suggested:
\begin{itemize}
	\item[(i)] {Extension of our method to other problems such as space fractional Schr\"{o}dinger equations
		         \cite{zhao2014fourth};} 
	\item[(ii)] {For solving higher-dimensional version of \eqref{eq1.1}, we suggest considering
		         the dynamical tensor approximation proposed in \cite{koch2010dynamical, lubich2013dynamical};}  
	\item[(iii)] {Design some fast implementations (e.g., a parallel version) of our method.}              
\end{itemize}
\section*{Acknowledgments}
\addcontentsline{toc}{section}{Acknowledgments}
\label{sec7}

\textit{This research is supported by the National Natural Science Foundation
of China (No.~11801463), the Applied Basic Research Project of Sichuan Province
(No.~2020YJ0007) and the Fundamental Research Funds for the Central Universities (No.~JBK1902028).}

\section*{References}
\addcontentsline{toc}{section}{References}
\bibliography{references}

\begin{thebibliography}{10}
\expandafter\ifx\csname url\endcsname\relax
  \def\url#1{\texttt{#1}}\fi
\expandafter\ifx\csname urlprefix\endcsname\relax\def\urlprefix{URL }\fi
\expandafter\ifx\csname href\endcsname\relax
  \def\href#1#2{#2} \def\path#1{#1}\fi

\bibitem{newell1969finite}
A.~C. Newell, J.~A. Whitehead, Finite bandwidth, finite amplitude convection,
  J. Fluid Mech. 38 (1969) 279--303.

\bibitem{lange1974stability}
C.~G. Lange, A.~C. Newell, A stability criterion for envelope equations, SIAM
  J. Appl. Math. 27 (1974) 441--456.

\bibitem{segel1969distant}
L.~A. Segel, Distant side-walls cause slow amplitude modulation of cellular
  convection, J. Fluid Mech. 38 (1969) 203--224.

\bibitem{stewartson1971non}
K.~Stewartson, J.~Stuart, A non-linear instability theory for a wave system in
  plane {P}oiseuille flow, J. Fluid Mech. 48 (1971) 529--545.

\bibitem{du1992analysis}
Q.~Du, M.~D. Gunzburger, J.~S. Peterson, Analysis and approximation of the
  {G}inzburg-{L}andau model of superconductivity, SIAM Rev. 34 (1992) 54--81.

\bibitem{chapman1992macroscopic}
S.~J. Chapman, S.~D. Howison, J.~R. Ockendon, Macroscopic models for
  superconductivity, SIAM Rev. 34 (1992) 529--560.

\bibitem{tarasov2005fractional}
V.~E. Tarasov, G.~M. Zaslavsky, Fractional {G}inzburg--{L}andau equation for
  fractal media, Physica A 354 (2005) 249--261.

\bibitem{tarasov2006fractional}
V.~E. Tarasov, G.~M. Zaslavsky, Fractional dynamics of coupled oscillators with
  long-range interaction, Chaos 16 (2006) 023110.
\newblock \href {http://dx.doi.org/10.1063/1.2197167}
  {\path{doi:10.1063/1.2197167}}.

\bibitem{milovanov2005fractional}
A.~V. Milovanov, J.~J. Rasmussen, Fractional generalization of the
  {G}inzburg--{L}andau equation: an unconventional approach to critical
  phenomena in complex media, Phys. Lett. A 337 (2005) 75--80.

\bibitem{mvogo2016localized}
A.~Mvogo, A.~Tambue, G.~H. Ben-Bolie, T.~C. Kofan{\'e}, Localized numerical
  impulse solutions in diffuse neural networks modeled by the complex
  fractional {G}inzburg--{L}andau equation, Commun. Nonlinear Sci. Numer.
  Simul. 39 (2016) 396--410.

\bibitem{tarasov2006psi}
V.~E. Tarasov, Psi-series solution of fractional {G}inzburg--{L}andau equation,
  J. Phys. A-Math. Theor. 39 (2006) 8395--8407.

\bibitem{pu2013well}
X.~Pu, B.~Guo, Well-posedness and dynamics for the fractional
  {G}inzburg-{L}andau equation, Appl. Anal. 92 (2013) 318--334.

\bibitem{guo2013well}
B.~Guo, Z.~Huo, Well-posedness for the nonlinear fractional {S}chr{\"o}dinger
  equation and inviscid limit behavior of solution for the fractional
  {G}inzburg-{L}andau equation, Fract. Calc. Appl. Anal. 16 (2013) 226--242.

\bibitem{lu2013asymptotic}
H.~Lu, S.~L{\"u}, Z.~Feng, Asymptotic dynamics of 2{D} fractional complex
  {G}inzburg--{L}andau equation, Int. J. Bifurcation Chaos 23 (2013) 1350202.
\newblock \href {http://dx.doi.org/10.1142/S0218127413502027}
  {\path{doi:10.1142/S0218127413502027}}.

\bibitem{millot2015fractional}
V.~Millot, Y.~Sire, On a fractional {G}inzburg--{L}andau equation and
  1/2-harmonic maps into spheres, Arch. Ration. Mech. Anal. 215 (2015)
  125--210.

\bibitem{wang2016implicit}
P.~Wang, C.~Huang, An implicit midpoint difference scheme for the fractional
  {G}inzburg--{L}andau equation, J. Comput. Phys. 312 (2016) 31--49.

\bibitem{hao2017linearized}
Z.-P. Hao, Z.-Z. Sun, A linearized high-order difference scheme for the
  fractional {G}inzburg--{L}andau equation, Numer. Meth. Part Differ. Equ. 33
  (2017) 105--124.

\bibitem{li2019efficient}
M.~Li, C.~Huang, An efficient difference scheme for the coupled nonlinear
  fractional {G}inzburg--{L}andau equations with the fractional {L}aplacian,
  Numer. Meth. Part. Differ. Equ. 35 (2019) 394--421.

\bibitem{zhang2020exponential}
L.~Zhang, Q.~Zhang, H.-W. Sun, Exponential {R}unge--{K}utta method for
  two-dimensional nonlinear fractional complex {G}inzburg--{L}andau equations,
  J. Sci. Comput. 83 (2020) 59.
\newblock \href {http://dx.doi.org/10.1007/s10915-020-01240-x}
  {\path{doi:10.1007/s10915-020-01240-x}}.

\bibitem{zhang2019linearized}
Z.~Zhang, M.~Li, Z.~Wang, A linearized {C}rank--{N}icolson {G}alerkin {FEM}s
  for the nonlinear fractional {G}inzburg--{L}andau equation, Appl. Anal. 98
  (2019) 2648--2667.

\bibitem{li2017galerkin}
M.~Li, C.~Huang, N.~Wang, Galerkin finite element method for the nonlinear
  fractional {G}inzburg--{L}andau equation, Appl. Numer. Math. 118 (2017)
  131--149.

\bibitem{wang2018efficient}
N.~Wang, C.~Huang, An efficient split-step quasi-compact finite difference
  method for the nonlinear fractional {G}inzburg--{L}andau equations, Comput.
  Math. Appl. 75 (2018) 2223--2242.

\bibitem{he2018unconditionally}
D.~He, K.~Pan, An unconditionally stable linearized difference scheme for the
  fractional {G}inzburg-{L}andau equation, Numer. Algorithms 79 (2018)
  899--925.

\bibitem{zhang2020linearized}
Q.~Zhang, X.~Lin, K.~Pan, Y.~Ren, Linearized {ADI} schemes for two-dimensional
  space-fractional nonlinear {G}inzburg--{L}andau equation, Comput. Math. Appl.
  80 (2020) 1201--1220.

\bibitem{koch2007dynamical}
O.~Koch, C.~Lubich, Dynamical low-rank approximation, SIAM J. Matrix Anal.
  Appl. 29 (2007) 434--454.

\bibitem{gorenflo1998random}
R.~Gorenflo, F.~Mainardi, Random walk models for space-fractional diffusion
  processes, Fract. Calc. Appl. Anal. 1 (1998) 167--191.

\bibitem{wang2015energy}
P.~Wang, C.~Huang, An energy conservative difference scheme for the nonlinear
  fractional {S}chr{\"o}dinger equations, J. Comput. Phys. 293 (2015) 238--251.

\bibitem{zhao2014fourth}
X.~Zhao, Z.-Z. Sun, Z.-P. Hao, A fourth-order compact {ADI} scheme for
  two-dimensional nonlinear space fractional {S}chr{\"o}dinger equation, SIAM
  J. Sci. Comput. 36 (2014) A2865--A2886.

\bibitem{wang2014linearly}
D.~Wang, A.~Xiao, W.~Yang, A linearly implicit conservative difference scheme
  for the space fractional coupled nonlinear {S}chr{\"o}dinger equations, J.
  Comput. Phys. 272 (2014) 644--655.

\bibitem{zhai2019error}
S.~Zhai, D.~Wang, Z.~Weng, X.~Zhao, Error analysis and numerical simulations of
  {S}trang splitting method for space fractional nonlinear {S}chr{\"o}dinger
  equation, J. Sci. Comput. 81 (2019) 965--989.

\bibitem{li2020relaxation}
M.~Li, C.~Huang, W.~Ming, A relaxation-type {G}alerkin {FEM} for nonlinear
  fractional {S}chr{\"o}dinger equations, Numer. Algorithms 83 (2020) 99--124.

\bibitem{dirac1930principles}
P.~A.~M. Dirac, The Principles of Quantum Mechanics, 4th Edition, Oxford
  University Press, Oxford, 1981.

\bibitem{lubich2008quantum}
C.~Lubich, From Quantum to Classical Molecular Dynamics: Reduced Models and
  Numerical Analysis, European Mathematical Society, Z\"{u}rich, 2008.

\bibitem{nonnenmacher2008dynamical}
A.~Nonnenmacher, C.~Lubich, Dynamical low-rank approximation: applications and
  numerical experiments, Math. Comput. Simul. 79 (2008) 1346--1357.

\bibitem{lubich2014projector}
C.~Lubich, I.~V. Oseledets, A projector-splitting integrator for dynamical
  low-rank approximation, BIT 54 (2014) 171--188.

\bibitem{kieri2016discretized}
E.~Kieri, C.~Lubich, H.~Walach, Discretized dynamical low-rank approximation in
  the presence of small singular values, SIAM J. Numer. Anal. 54 (2016)
  1020--1038.

\bibitem{ostermann2019convergence}
A.~Ostermann, C.~Piazzola, H.~Walach, Convergence of a low-rank {L}ie-{T}rotter
  splitting for stiff matrix differential equations, SIAM J. Numer. Anal. 57
  (2019) 1947--1966.

\bibitem{koch2010dynamical}
O.~Koch, C.~Lubich, Dynamical tensor approximation, SIAM J. Matrix Anal. Appl.
  31 (2010) 2360--2375.

\bibitem{lubich2013dynamical}
C.~Lubich, T.~Rohwedder, R.~Schneider, B.~Vandereycken, Dynamical approximation
  by hierarchical {T}ucker and tensor-train tensors, SIAM J. Matrix Anal. Appl.
  34 (2013) 470--494.

\bibitem{einkemmer2018low}
L.~Einkemmer, C.~Lubich, A low-rank projector-splitting integrator for the
  {V}lasov-{P}oisson equation, SIAM J. Sci. Comput. 40 (2018) B1330--B1360.

\bibitem{einkemmer2019quasi}
L.~Einkemmer, C.~Lubich, A quasi-conservative dynamical low-rank algorithm for
  the {V}lasov equation, SIAM J. Sci. Comput. 41 (2019) B1061--B1081.

\bibitem{einkemmer2019low}
L.~Einkemmer, A low-rank algorithm for weakly compressible flow, SIAM J. Sci.
  Comput. 41 (2019) A2795--A2814.

\bibitem{grasedyck2013literature}
L.~Grasedyck, D.~Kressner, C.~Tobler, A literature survey of low-rank tensor
  approximation techniques, GAMM-Mitt. 36 (2013) 53--78.

\bibitem{ccelik2012crank}
C.~{\c{C}}elik, M.~Duman, Crank--{N}icolson method for the fractional diffusion
  equation with the {R}iesz fractional derivative, J. Comput. Phys. 231 (2012)
  1743--1750.

\bibitem{al2011computing}
A.~H. Al-Mohy, N.~J. Higham, Computing the action of the matrix exponential,
  with an application to exponential integrators, SIAM J. Sci. Comput. 33
  (2011) 488--511.

\bibitem{caliari2016leja}
M.~Caliari, P.~Kandolf, A.~Ostermann, S.~Rainer, The {L}eja method revisited:
  {B}ackward error analysis for the matrix exponential, SIAM J. Sci. Comput. 38
  (2016) A1639--A1661.

\bibitem{saad1992analysis}
Y.~Saad, Analysis of some {K}rylov subspace approximations to the matrix
  exponential operator, SIAM J. Numer. Anal. 29 (1992) 209--228.

\bibitem{lee2010shift}
S.~T. Lee, H.-K. Pang, H.-W. Sun, Shift-invert {A}rnoldi approximation to the
  {T}oeplitz matrix exponential, SIAM J. Sci. Comput. 32 (2010) 774--792.

\bibitem{helmke1996optimization}
U.~Helmke, J.~B. Moore, Optimization and Dynamical Systems, Springer-Verlag,
  London, 1994.

\bibitem{sun2016some}
H.~Sun, Z.-Z. Sun, G.-H. Gao, Some high order difference schemes for the space
  and time fractional {B}loch–{T}orrey equations, Appl. Math. Comput. 281
  (2016) 356--380.

\end{thebibliography}

\end{document}